\documentclass[reqno]{amsart}
\usepackage{amsmath,amssymb,amsfonts}
   \DeclareMathOperator{\Id}{Id}
   \DeclareMathOperator{\e}{e}
\usepackage{dsfont}
   \newcommand{\N}{\ensuremath{\mathds N}}
   \newcommand{\Z}{\ensuremath{\mathds Z}}
   \newcommand{\R}{\ensuremath{\mathds R}}
\usepackage{euscript}
   \newcommand{\cA}{\EuScript{A}}
   \newcommand{\hcA}{\widehat{\EuScript{A}}}

\usepackage{enumitem}
   \newcommand{\lb}{label}
   \newcommand{\lm}{leftmargin}

\newtheorem{theorem}{Theorem}
\newtheorem{lemma}[theorem]{Lemma}

\newtheorem{corollary}[theorem]{Corollary}

\newcommand{\eps}{\varepsilon}
\newcommand{\lbd}{\lambda}
\newcommand{\Omg}{\Omega}
\newcommand{\prts}[1]{\left(#1\right)}

\newcommand{\norm}[1]{\left\|#1\right\|}

\newcommand{\set}[1]{\left\{#1\right\}}

\newcommand{\maxs}[1]{\max\set{#1}}
\newcommand{\mins}[1]{\min\set{#1}}
\newcommand{\sups}[1]{\sup\set{#1}}

\newcommand{\pfrac}[2]{\prts{\dfrac{#1}{#2}}}
\newcommand{\dsum}{\displaystyle\sum}
\newcommand{\dsup}{\displaystyle\sup}
\renewcommand{\ge}{\geqslant}

\renewcommand{\le}{\leqslant}

\begin{document}
\title[Robusteness of discrete nonuniform dichotomic behavior]
   {Robusteness of discrete nonuniform dichotomic behavior}
\author[Ant\'onio J. G. Bento]{Ant\'onio J. G. Bento}
\address{
   Ant\'onio J. G. Bento\\
   Departamento de Matem\'atica\\
   Universidade da Beira Interior\\
   6201-001 Covilh\~a\\
   Portugal}
\email{bento@ubi.pt}
\author{C\'esar M. Silva}
\address{
   C\'esar M. Silva\\
   Departamento de Matem\'atica\\
   Universidade da Beira Interior\\
   6201-001 Covilh\~a\\
   Portugal}
\email{csilva@ubi.pt}
\urladdr{www.mat.ubi.pt/~csilva}
\date{\today}
\begin{abstract}
   For nonautonomous linear difference equations in Banach spaces we show that a very general type of dichotomic behavior persists under small enough additive linear perturbations. By using a new approach, we obtain two general robustness theorems that improve several results in the literature and also contain new situations. In particular, unlike several existent results for particular growth rates, we show that, up to a multiplicative constant, the dichotomic behavior for the perturbed equation is the same as the one for the original equation.
\end{abstract}
\maketitle
\section{Introduction}
The notion of exponential dichotomy, that can be traced back to the work of Perron~\cite{Perron-MZ-1930} on the stability of ordinary differential equations, and of Li~\cite{Li-AM-1934} for discrete time systems, is a fundamental tool in the study of stability of difference and differential equations. Thus, the so-called robustness (or roughness) problem for these dichotomies, that is the problem of knowing if small additive linear perturbations of a linear system with some type of dichotomy still possesses the same type of dichotomy, is a fundamental problem. For nonautonomous linear differential equations this problem was already discussed in 1958 by Massera and Sch\"affer~\cite{Massera-Schaffer-AM-1958}. Several other authors, among which
Coppel~\cite{Coppel-JDE-1967}, Dalec'kii and Krein~\cite{Daleckii-Krein-TMM-1974}, Chow and Leiva~\cite{Chow-Leiva-JDE-1995}, Pliss and Sell~\cite{Pliss-Sell-JDDE-1999} and Popescu~\cite{Popescu-JMAA-2006} studied the problem of robustness under different assumptions and in different settings.

In spite of its usefulness in the nonautonomous setting, the notion of exponential dichotomy is sometimes restrictive and therefore more general dichotomic behavior have been considered. There is a large number os papers studying different aspects of the dynamics associated with nonuniform exponential dichotomies, a type of exponential dichotomic behavior where some exponential loss of hyperbolicity along the trajectories is allowed. For example, the problem of existence of nonuniform exponential dichotomies via admissibility properties was discussed by Preda and Megan~\cite{Preda-Megan-BAusMS-1983} and the existence of invariant stable manifolds for nonuniform exponential dichotomies was proved by Barreira and Valls in~\cite{Barreira-Valls-CMP-2005, Barreira-Valls-JDE-2006-221}.
Robustness results for nonuniform exponential dichotomies were obtained by Barreira and Valls~\cite{Barreira-Valls-JFA-2007, Barreira-Valls-JDE-2008} in the continuous time setting and by Barreira, Silva and Valls~\cite{Barreira-Silva-Valls-JDE-2009} for discrete time.

We remark that, unlike the autonomous measure-preserving situation where it is known that almost all linear variational equations obtained from a measure-preserving flow on a smooth Riemannian manifold admit a nonuniform exponential dichotomy with arbitrarily small nonuniformity, for general nonautonomous differential equations there is no reason to assume that nonuniform exponential dichotomies are typical. Recently, there were several papers dedicated to the study of the behavior of trajectories of linear perturbations of nonlinear nonautonomous systems, assuming that there is a nonuniform and nonexponential dichotomy for the linear part or assuming that the linear part belongs to a family that contains the exponential behavior as a special case. Several results were obtained for this nonuniform nonexponential dichotomies. For instance the existence of stable manifolds for perturbations of linear equations with nonuniform polynomial dichotomies and, more generally, for the so-called nonuniform $(\mu,\nu)$-dichotomies was discussed by Bento and Silva~\cite{Bento-Silva-JFA-2009,Bento-Silva-NATMA-2012} and the existence of polynomial trichotomies based on the values of some generalized “polynomial” Lyapunov exponents was established by Barreira and Valls~\cite{Barreira-Valls-NATMA-2010-72}. We should also mention that in the nineties, Naulin and Pinto~\cite{Naulin-Pinto-NATMA-1994, Naulin-Pinto-JDE-1995} obtained stability and robustness results for nonexponential dichotomies, although in the uniform setting. More recently, robustness results for nonuniform polynomial dichotomies were derived by Barreira, Fan, Valls and Zhang~\cite{Barreira-Fan-Valls-Zhang-TMNA-2011}, for nonuniform $(\mu,\nu)$-dichotomies by Chang, Zhang, and Qin~\cite{Chang-Zhang-Qin-JMAA-2012} and Chu~\cite{Chu-BSM-2013} and for nonuniform $(\mu,\nu)$-trichotomies by Jiang~\cite{Jiang-EJDE-2012}.

Though several nontrivial problems must be solved when one tries to prove the same type of result for distinct nonuniform growth rates, several common features can be found in the proofs. With this in mind the authors decided to propose a general setting where the above results may be obtained for the different types of nonuniform behavior all at the same time and undertook the project of obtaining several results in this general setting. The present paper is part of this project that already motivated the previous works~\cite{Bento-Silva-2012-arXiv1209.6589B, Bento-Silva-2012-arXiv1210.7740B}, dedicated to the existence of Lipschitz stable manifolds for (nonlinear and nonautonomous) perturbations of nonautonomous linear equations with this general nonuniform behavior. Our main result in this paper is a robustness (or roughness) result for our general dichotomic behaviors, i.e., a result that establishes that a very large class of dichotomic behaviours are persistent under small linear perturbations of the dynamics. More precisely, it is shown that if the equation
\begin{equation}\label{eq:x_m+1=A_m x_m, m in Z}
   x_{n+1} = A_n x_n, \ n \in \Z,
\end{equation}
where, for every $n \in \Z$, $A_n$ is a bounded linear operator from a Banach space $X$ into $X$,
has some general nonuniform dichotomic behaviour (in some sense to be clarified later) then the equation
\begin{equation}\label{eq:x_m+1=(A_m+B_m)x_m, m in Z}
   x_{n+1} = (A_n + B_n) x_n, \ n \in \Z
\end{equation}
admits also the same dichotomic behavior, where $B_n$ is a bounded linear operator from $X$ into $X$ and provided that $\|B_n\|$ is sufficiently small and decays to zero sufficiently fast as $n$ goes to infinity. Next, based on this result we derive a corresponding result in $\N$.

By considering this general setting we intend to achieve the following goals: firstly, we want to unify the several settings in the literature and consider a general situation that is sufficiently flexible to include a wide range of nonuniform behaviors; secondly, we want to highlight the common aspects of the arguments used in the proofs and to show that, in this general setting, there is no central role played by exponential or polynomial behavior and instead the same type of result holds for a very large family of prescribed behaviors for the dichotomy; finally we intend to use a different approach in the proof that allows our theorems not only to include new situations but also to improve existent results.

We emphasise that our proof is considerably different from the proofs of the robustness results that constituted our departing point. In fact, instead of using two fixed points to obtain the new dichotomy we use only one fixed point (see Lemma~\ref{lemma:U_m,n_&_V_m,n}) to get simultaneously the behavior in the subspaces associated to the family of projections and in the complementary subspaces. In particular, we have no need to compute bounds for the projections separately and this allowed us to obtain a dichotomic behavior for the perturbed equation that is asymptotically the same as the one assumed for the unperturbed equation. This contrasts with existent robustness results for nonuniform exponential dichotomies, nonuniform polynomial dichotomies and nonuniform $(\mu,\nu)$-dichotomies where the bounds for the nonuniform part of the dichotomy of the perturbed equation is asymptotically different from the bounds for the nonuniform part of the original one. In several corollaries we use our main results to obtain improved versions of the robustness theorems in~\cite{Barreira-Fan-Valls-Zhang-TMNA-2011, Barreira-Silva-Valls-JDE-2009, Barreira-Valls-DCDS-A-2012-32-12, Chu-BSM-2013} as well as some new cases.

The structure of the paper is the following: in Section~\ref{Section:Robustness:Z} and~\ref{Section:Robustness:N} we state the our robustness results, in $\Z$ and $\N$, respectively, and we apply them to several growth rates to improve some known results and obtain new ones and in Sections~\ref{section:proof:robustness:Z} and~\ref{section:proof:robustness:N} we prove our main results.
\section{Robustness in $\Z$} \label{Section:Robustness:Z}
Let $\Z$ be the set of integer numbers and define $\Z^2_\ge =\set{\prts{m,n} \in \Z^2 \colon m \ge n}$, $\Z^2_> = \set{\prts{m,n} \in \Z^2 \colon m > n}$, $\Z^2_\le = \Z^2 \setminus \Z^2_>$ and $\Z^2_< = \Z^2 \setminus \Z^2_\ge$.

Let $B(X)$ be the space of bounded linear operators in a Banach space~$X$. Given a sequence $(A_n)_{n \in \Z}$ of operators of $B(X)$, we write, for every $\prts{m,n} \in \Z^2_\ge$,
   $$ \cA_{m,n} =
      \begin{cases}
         A_{m-1} \cdots A_n & \text{ if } m>n, \\
         \text{Id} & \text{ if } m=n.\\
      \end{cases}$$

We say that the linear difference equation~\eqref{eq:x_m+1=A_m x_m, m in Z} admits an \textit{invariant splitting} if there exist bounded projections $P_n$, $n\in\Z$, such that, for every $(m,n) \in \Z_\ge^2$, we have
\begin{enumerate}[\lb=$($S$\arabic*)$,\lm=13mm]
   \item $ P_m \cA_{m,n} = \cA_{m,n} P_n$; \label{eq:(S1)}
   \item $\cA_{m,n}(\ker P_n) = \ker P_m$, where $\ker P_n$ and $\ker P_m$ are the null spaces of $P_n$ and $P_m$, respectively; \label{eq:(S2)}
   \item $\cA_{m,n}|_{\ker P_n} \colon \ker P_n \to \ker P_m$ is invertible in $B(X)$. \label{eq:(S3)}
\end{enumerate}
In these conditions, we denote by $\cA_{n,m}$ the inverse of $\cA_{m,n}|_{\ker P_n} \colon \ker P_n \to \ker P_m$, $\prts{m,n} \in \Z^2_\ge$, and we define, for each $n \in \Z$, the complementary projection $Q_n=\Id-P_n$ and the linear subspaces $E_n=P_n (X)$ and $F_n= \ker P_n = Q_n(X)$. As usual, we identify the vector spaces $E_n \times F_n$ and $E_n \oplus F_n$ as the same vector space.

Given double sequences $(a_{m,n})_{(m,n) \in \Z_\ge^2}$ and $(b_{m,n})_{(m,n) \in \Z_\le^2}$ we say that equation~\eqref{eq:x_m+1=A_m x_m, m in Z} admits a \textit{general dichotomy with bounds $\prts{a_{m,n}}$ and $\prts{b_{m,n}}$} if it admits an invariant splitting such that
\begin{enumerate}[\lb=$($D$\arabic*)$,\lm=13mm]
   \item $\|\cA_{m,n} P_n\| \le a_{m,n}$ for every $\prts{m,n} \in \Z^2_\ge$ and \label{eq:(D1)}
   \item $\|\cA_{m,n} Q_n\| \le b_{m,n}$ for every $\prts{m,n} \in \Z^2_\le$. \label{eq:(D2)}
\end{enumerate}

Given a sequence $\prts{B_n}_{n \in \Z} \subseteq B(X)$, we will consider the perturbed difference equation~\eqref{eq:x_m+1=(A_m+B_m)x_m, m in Z} and, for this equation, using the convention $\dsum_{k=n}^{n-1} \alpha_k=0$,  we define
\begin{equation}\label{def:lbd_m,n}
   \lbd_{m,n}
   = \dsum_{k=-\infty}^{n-1} a_{m,k+1} \|B_k\| b_{k,n}
      + \dsum_{k=n}^{m-1} a_{m,k+1} \|B_k\| a_{k,n}
      + \dsum_{k=m}^{+\infty} b_{m,k+1} \|B_k\| a_{k,n}
\end{equation}
for every $\prts{m,n} \in \Z^2_\ge$ and
\begin{equation}\label{def:mu_m,n}
   \mu_{m,n}
   = \dsum_{k=-\infty}^{m-1} a_{m,k+1} \|B_k\| b_{k,n}
      + \dsum_{k=m}^{n-1} b_{m,k+1} \|B_k\| b_{k,n}
      + \dsum_{k=n}^{+\infty} b_{m,k+1} \|B_k\| a_{k,n},
\end{equation}
for every $\prts{m,n} \in \Z^2_\le$ and we assume that
\begin{equation}\label{def:lbd+mu}
   \lbd
   = \dsup_{(m,n) \in \Z^2_\ge} \dfrac{\lbd_{m,n}}{a_{m,n}}
   < + \infty
   \ \ \ \text{ and } \ \ \
   \mu
   = \dsup_{(m,n) \in \Z^2_\le} \dfrac{\mu_{m,n}}{b_{m,n}}
   < + \infty.
\end{equation}

The main theorem of this section is the following:

\begin{theorem} \label{thm:robustness:Z}
   Suppose that equation~\eqref{eq:x_m+1=A_m x_m, m in Z} admits a generalized dichotomy with bounds $\prts{a_{m,n}}$ and $\prts{b_{m,n}}$  such that
   \begin{equation}\label{hyp:sup_m_a_m,n/a_m,j<infty}
      \sup_{m \ge j} \dfrac{a_{m,n}}{a_{m,j}} < +\infty
      \text{ for every } \prts{j,n} \in \Z^2_\ge
   \end{equation}
   and
   \begin{equation}\label{hyp:sup_m_b_m,n/b_m,j<infty}
      \sup_{m \le j} \dfrac{b_{m,n}}{b_{m,j}} < +\infty
      \text{ for every } \prts{j,n} \in \Z^2_\le
   \end{equation}
   and let $\prts{B_n}_{n \in \Z} \subset B(X)$ be a sequence of operators. If
   \begin{equation}\label{ine:max(lbd,mu)<1}
      \maxs{\lbd,\mu} < 1,
   \end{equation}
   where $\lbd$  and $\mu$ are defined by~\eqref{def:lbd+mu}, then equation~\eqref{eq:x_m+1=(A_m+B_m)x_m, m in Z} admits a generalized dichotomy with bounds $\prts{\sigma a_{m,n}}$ and $\prts{\sigma b_{m,n}}$,
   where $\sigma = 1/\prts{1-\maxs{\lbd,\mu}}$.
\end{theorem}

The proof of this theorem will be given in Section~\ref{section:proof:robustness:Z}.

Now we are going to apply Theorem~\ref{thm:robustness:Z} to general dichotomies with bounds of the form
\begin{equation}\label{a_m,n=a_n/a_m c_n...}
   a_{m,n} = \dfrac{a_n}{a_m} c_n,
         \ \ \ \text{ and } \ \ \
         b_{m,n} = \dfrac{b_n}{b_m} c_n,
\end{equation}
where $\prts{a_n}_{n \in \Z}$, $\prts{b_n}_{n \in \Z}$ and $\prts{c_n}_{n \in \Z}$ are sequences of positive numbers. We will refer to the sequences $(a_n/a_m)$ and $(b_n/b_m)$ as the \textit{uniform bounds of the dichotomy} and the sequence $(c_n)$ as the \textit{nonuniform bound of the dichotomy}.

\begin{corollary} \label{cor:robustness:Z:a_m,n=a_n/b_n c_n}
   Suppose that equation~\eqref{eq:x_m+1=A_m x_m, m in Z} admits a dichotomy with bounds of the form~\eqref{a_m,n=a_n/a_m c_n...} and
   \begin{equation}\label{hyp:a_m b_n/(a_n b_m) <= C:in Z}
      \dfrac{a_m b_n}{a_n b_m} \le C
         \text{ for every } \prts{m,n} \in \Z^2_\le
   \end{equation}
   for some positive $C$. Let $\prts{B_n}_{n \in \Z} \subset B(X)$ be a sequence of bounded linear operators such that
      $$ \|B_n\| \le \dfrac{a_n}{a_{n+1} c_{n+1}} \beta_n
         \text{ for every } n \in \Z,$$
   where $(\beta_n)_{n \in \N}$ is a sequence of positive numbers such that $\beta:=\dsum_{k=-\infty}^{+\infty} \beta_k < +\infty$. If $\beta C < 1$, then~\eqref{eq:x_m+1=(A_m+B_m)x_m, m in Z} admits a dichotomy with bounds obtained by multiplying the bounds in~\eqref{a_m,n=a_n/a_m c_n...} by $1/(1-\beta C)$.
\end{corollary}

\begin{proof}
   Clearly, this type of bounds satisfy hypothesis~\eqref{hyp:sup_m_a_m,n/a_m,j<infty} and~\eqref{hyp:sup_m_b_m,n/b_m,j<infty}. Moreover, since from~\eqref{hyp:a_m b_n/(a_n b_m) <= C:in Z} we have $C \ge 1$, we conclude that
   \begin{align*}
      \lbd_{m,n}
      & = \dsum_{k=-\infty}^{n-1} a_{m,k+1} \|B_k\| b_{k,n}
         + \dsum_{k=n}^{m-1} a_{m,k+1} \|B_k\| a_{k,n}
         + \dsum_{k=m}^{+\infty} b_{m,k+1} \|B_k\| a_{k,n}\\
      & \le \dsum_{k=-\infty}^{n-1} \dfrac{a_{k+1}}{a_m} c_{k+1}
            \dfrac{a_k}{a_{k+1} c_{k+1}} \beta_k \dfrac{b_n}{b_k} c_n
         + \dsum_{k=n}^{m-1} \dfrac{a_{k+1}}{a_m} c_{k+1}
            \dfrac{a_k}{a_{k+1} c_{k+1}} \beta_k \dfrac{a_n}{a_k} c_n\\
      & \qquad +\dsum_{k=m}^{+\infty} \dfrac{b_{k+1}}{b_m} c_{k+1}
         \dfrac{a_k}{a_{k+1} c_{k+1}} \beta_k \dfrac{a_n}{a_k} c_n\\
      & \le \dfrac{a_n}{a_m} c_n \beta C
   \end{align*}
   and
   \begin{align*}
      \mu_{m,n}
      & = \dsum_{k=-\infty}^{m-1} a_{m,k+1} \|B_k\| b_{k,n}
         + \dsum_{k=m}^{n-1} b_{m,k+1} \|B_k\| b_{k,n}
         + \dsum_{k=n}^{+\infty} b_{m,k+1} \|B_k\| a_{k,n}\\
      & \le \dsum_{k=-\infty}^{m-1} \dfrac{a_{k+1}}{a_m} c_{k+1}
            \dfrac{a_k}{a_{k+1} c_{k+1}} \beta_k \dfrac{b_n}{b_k} c_n
         + \dsum_{k=m}^{n-1} \dfrac{b_{k+1}}{b_m} c_{k+1}
            \dfrac{a_k}{a_{k+1} c_{k+1}} \beta_k \dfrac{b_n}{b_k} c_n\\
      & \qquad +\dsum_{k=n}^{+\infty} \dfrac{b_{k+1}}{b_m} c_{k+1}
         \dfrac{a_k}{a_{k+1} c_{k+1}} \beta_k \dfrac{a_n}{a_k} c_n\\
      & \le \dfrac{b_n}{b_m} c_n \beta C,
   \end{align*}
   and thus
      $$ \maxs{\lbd,\mu} \le \beta C.$$
   Hence, if $\beta C < 1$, all the hypothesis in Theorem~\ref{thm:robustness:Z} are satisfied and the result follows immediately.
\end{proof}

Now we will apply our result to nonuniform exponential dichotomies, i.e., dichotomies with bounds of the form
\begin{equation}\label{a_m,n=e^(a(m-n)+eps|n|)...}
   a_{m,n} = D\e^{a(m-n)+\eps |n|}
   \ \ \ \text{ and } \ \ \
   b_{m,n} = D \e^{b(m-n) + \eps |n|},
\end{equation}
where $D \ge 1$, $\eps \ge 0$ and $a,b \in \R$.

\begin{corollary}\label{cor:bonustness:Z:exponential}
   Suppose that equation~\eqref{eq:x_m+1=A_m x_m, m in Z} has a nonuniform exponential dichotomy, i.e, admits a general dichotomy with bounds given by~\eqref{a_m,n=e^(a(m-n)+eps|n|)...} and let $B_n \colon X \to X$, $n \in \Z$, be bounded linear perturbations such that
   \begin{equation*}
      \|B_n\| \le \delta \e^{-\gamma |n+1|}.
   \end{equation*}
   for some $\delta > 0$ and some $\gamma \in \R$. If
      $$ a \le b, \ \ \
         \eps - \gamma < 0 \ \ \ \text{ and } \ \ \
         \theta:= D \delta \e^{-a} (\e^{\eps-\gamma}+1) \prts{1-\e^{\eps-\gamma}}^{-1}
            < 1,$$
   then equation~\eqref{eq:x_m+1=(A_m+B_m)x_m, m in Z} admits a nonuniform exponential dichotomy with bounds of the form~\eqref{a_m,n=e^(a(m-n)+eps|n|)...} with $D$ replaced by $D/(1-\theta)$.
\end{corollary}

\begin{proof}
   Let
      $$ a_n = \e^{-an}, \ \ \
         b_n = \e^{-bn}, \ \ \
         c_n = D \e^{\eps |n|} \ \ \ \text{ and } \ \ \
         \beta_n = D \delta \e^{-a} \e^{-(\gamma-\eps) |n+1|}.$$
   Then
      $$ \dfrac{a_n}{a_m} c_n = D\e^{a(m-n)+\eps |n|}, \ \ \
         \dfrac{b_n}{b_m} c_n = D\e^{b(m-n)+\eps |n|} \ \ \ \text{ and } \ \ \
         \dfrac{a_n}{a_{n+1} c_{n+1}} \beta_n = \delta \e^{-\gamma |n+1|}$$
   Moreover, since $a \le b$,~condition~\eqref{hyp:a_m b_n/(a_n b_m) <= C:in Z} is satisfied with $C=1$. Therefore, applying Corollary~\ref{cor:robustness:Z:a_m,n=a_n/b_n c_n} and using the fact that
      $$ \beta
         = \dsum_{k=-\infty}^{+\infty} \beta_k
         = \dfrac{D\delta\prts{\e^{\eps-\gamma}+1}}
            {\e^a \prts{1-\e^{\eps-\gamma}}},$$
   the result follows immediately.
\end{proof}

Corollary~\ref{cor:bonustness:Z:exponential} is an improvement of Theorem 8 of Barreira, Silva and Valls~\cite{Barreira-Silva-Valls-JDE-2009} and of Theorem 1 of Barreira and Valls~\cite{Barreira-Valls-DCDS-A-2012-32-12}. In~\cite{Barreira-Silva-Valls-JDE-2009} in order to obtain Theorem 8 the authors work with invertible cocycles, impose the three additional hypotheses
   $$ a=-b <0, \ \ \
      \eps < -a
      \ \ \ \text{ and } \ \ \
      \gamma = -2 \eps
      $$
and in the bounds of the dichotomy of the perturbed equation they have to replace $\eps$ by $2\eps$ and $a$ by $ \tilde a$, with $\tilde a > a$. In~\cite{Barreira-Valls-DCDS-A-2012-32-12} the authors proved the result for noninvertible cocycles, but used another three hypotheses
   $$ a=-b <0, \ \ \
      \eps < -2a
      \ \ \ \text{ and } \ \ \
      \gamma = -2 \eps$$
and the nonuniform bounds of the dichotomy of the perturbed equation have exponent $2\eps$ instead of $\eps$.

We can also obtain a new result considering polynomial growth rates similar to the ones in~\cite{Bento-Silva-JFA-2009}.

\begin{corollary}\label{cor:robustness:Z:a_m,n=D(m-n+1)^a (|n|+1)^eps...}
   Suppose that equation~\eqref{eq:x_m+1=A_m x_m, m in Z} has a nonuniform polynomial dichotomy, i.e, admits a general dichotomy with bounds given by
   \begin{equation}\label{a_m,n=D(m-n+1)^a (|n|+1)^eps...}
      a_{m,n} = D \prts{m-n+1}^a \prts{|n|+1}^\eps
      \ \ \ \text{ and } \ \ \
      b_{m,n} = D \prts{n-m+1}^b \prts{|n|+1}^\eps,
   \end{equation}
   with $D \ge 1$, $\eps \ge 0$ and $a,b \in \R$. Let $B_n \colon X \to X$, $n \in \Z$, be bounded linear perturbations such that
      $$ \|B_n\| \le \delta \prts{|n+1|+1}^{-\gamma}.$$
   for some $\delta > 0$ and some $\gamma \in \R$. If
      $$ a \le b \le 0, \ \ \
         \eps - \gamma < -1 \ \ \ \text{ and } \ \ \
         \theta:=2^{-a} D \delta \prts{2\zeta_{\gamma-\eps}-1} < 1,$$
   where $\zeta_{\gamma-\eps} := \dsum_{k=1}^{+\infty} k^{\eps-\gamma}$, then equation~\eqref{eq:x_m+1=(A_m+B_m)x_m, m in Z} admits a nonuniform polynomial dichotomy with bounds given by~\eqref{a_m,n=D(m-n+1)^a (|n|+1)^eps...} with $D$ replaced by $D/(1-\theta)$.
\end{corollary}

\begin{proof}
   Since $ a \le b \le 0$, for every integers $k < n \le m$ we have
      $$ (m-k)^a \le (m-n-1)^a \ \ \ \text{ and } \ \ \ (n-k+1)^b \le 1,$$
   for every integers $n \le k < m$ we have
      $$ (m-k)^a (k-n+1)^a \le (m-n)^a \le 2^{-a} (m-n+1)^a$$
   and for every integers $n \le m \le k$ we have
      $$ (k-m+2)^b \le 1 \ \ \ \text{ and } \ \ \ (k-n+1)^a \le (m-n+1)^a,$$
   it follows that
   \begin{align*}
      & \lbd_{m,n}\\
      & = \dsum_{k=-\infty}^{n-1} a_{m,k+1} \|B_k\| b_{k,n}
         + \dsum_{k=n}^{m-1} a_{m,k+1} \|B_k\| a_{k,n}
         + \dsum_{k=m}^{+\infty} b_{m,k+1} \|B_k\| a_{k,n}\\
      & \le \dsum_{k=-\infty}^{n-1} D (m-k)^a (|k+1|+1)^\eps
            \delta (|k+1|+1)^{-\gamma}
            D (n-k+1)^b (|n|+1)^\eps\\
      & \qquad + \dsum_{k=n}^{m-1} D (m-k)^a (|k+1|+1)^\eps
            \delta (|k+1|+1)^{-\gamma}
            D (k-n+1)^a (|n|+1)^\eps\\
      & \qquad + \dsum_{k=m}^{+\infty} D (k-m+2)^b (|k+1|+1)^\eps
            \delta (|k+1|+1)^{-\gamma}
            D (k-n+1)^a (|n|+1)^\eps\\
      & \le 2^{-a} D^2 \delta (m-n+1)^a (|n|+1)^\eps
         \dsum_{k=- \infty}^{+\infty} (|k+1|+1)^{\eps-\gamma}\\
      & = 2^{-a} D^2 \delta (m-n+1)^a (|n|+1)^\eps  (2 \zeta_{\gamma-\eps} -1).
   \end{align*}
   Analogously, we also have
      $$ \mu_{m,n}
         \le D^2 \delta (n-m+1)^b (|n|+1)^\eps  (2 \zeta_{\gamma-\eps} -1)$$
   Therefore, we have
      $$ \maxs{\lbd,\mu} \le 2^{-a} D \delta (2 \zeta_{\gamma-\eps} -1)$$
   and if $2^{-a} D \delta (2 \zeta_{\gamma-\eps} -1) < 1$, taking into account that these type of bounds also satisfy hypothesis~\eqref{hyp:sup_m_a_m,n/a_m,j<infty} and~\eqref{hyp:sup_m_b_m,n/b_m,j<infty}, Corollary~\ref{cor:robustness:Z:a_m,n=D(m-n+1)^a (|n|+1)^eps...} follows immediately from Theorem~\ref{thm:robustness:Z}.
\end{proof}
\section{Robustness in $\N$} \label{Section:Robustness:N}
In this section we are going study the robustness problem for a linear difference equation of the form
\begin{equation}\label{eq:x_m+1=A_m x_m, m in N}
   x_{m+1} = A_m x_m, \ m \in \N,
\end{equation}
with $\prts{A_m}_{m \in \N} \subset B(X)$. It is obvious that for this equation we can also define a general dichotomy with bounds $\prts{a_{m,n}}_{(m,n) \in \N^2_\ge}$ and $\prts{b_{m,n}}_{(m,n) \in \N^2_\le}$ in an analogous way to that of equation~\eqref{eq:x_m+1=A_m x_m, m in Z}, replacing $\Z$, $\Z^2_\ge$ and $\Z^2_\le$ by $\N$, $\N^2_\ge$ and $\N^2_\le$, respectively, where $\N^2_\ge= \N^2 \cap \Z^2_\ge$ and $\N^2_\le= \N^2 \cap \Z^2_\le$.

Given a sequence $\prts{B_n}_{n \in \N} \subset B(X)$, we consider equation
\begin{equation}\label{eq:x_m+1=(A_m+B_m)x_m, m in N}
   x_{m+1} = (A_m + B_m) x_m, \ m \in \N
\end{equation}
and we define
\begin{equation}\label{def:lbd'_m,n}
   \lbd'_{m,n}
   := \dsum_{k=1}^{n-1} a_{m,k+1} \|B_k\| b_{k,n}
      + \dsum_{k=n}^{m-1} a_{m,k+1} \|B_k\| a_{k,n}
      + \dsum_{k=m}^{+\infty} b_{m,k+1} \|B_k\| a_{k,n}
\end{equation}
for every $\prts{m,n} \in \N^2_\ge$ and
\begin{equation}\label{def:mu'_m,n}
   \mu'_{m,n}
   := \dsum_{k=1}^{m-1} a_{m,k+1} \|B_k\| b_{k,n}
      + \dsum_{k=m}^{n-1} b_{m,k+1} \|B_k\| b_{k,n}
      + \dsum_{k=n}^{+\infty} b_{m,k+1} \|B_k\| a_{k,n}
\end{equation}
for every $\prts{m,n} \in \N^2_\le$.

Our main result in this section is the following:

\begin{theorem}\label{thm:robustness:N}
   Suppose that equation~\eqref{eq:x_m+1=A_m x_m, m in N} admits a dichotomy with bounds $\prts{a_{m,n}}$ and $\prts{b_{m,n}}$ such that
   \begin{equation}\label{hyp:sup_m_a_m,n/a_m,j<infty:N}
      \sup\limits_{m \ge j} \dfrac{a_{m,n}}{a_{m,j}} < +\infty
      \text{ for every $\prts{j,n} \in \N^2_\ge$}.
   \end{equation}
   Let $\prts{B_n}_{n \in \N} \subset B(X)$ be a sequence of bounded linear operators such that
         $$ \theta:= \maxs{\sup\limits_{\prts{m,n} \in \N^2_\ge}
               \dfrac{\lbd'_{m,n}}{a_{m,n}}, \
               \sup\limits_{\prts{m,n} \in \N^2_\le}
               \dfrac{\mu'_{m,n}}{b_{m,n}}}
            < 1,$$
   where the $\lbd'$s and $\mu'$s are given by~\eqref{def:lbd'_m,n} and~\eqref{def:mu'_m,n}. Then equation~\eqref{eq:x_m+1=(A_m+B_m)x_m, m in N}
   admits a dichotomy with bounds $\prts{\sigma a_{m,n}}$ and $\prts{\sigma b_{m,n}}$, where $\sigma=1/(1-\theta)$.
\end{theorem}

The proof of Theorem~\ref{thm:robustness:N} will be given in section~\ref{section:proof:robustness:N}.

Now we will apply Theorem~\ref{thm:robustness:N} to dichotomies with bounds of the form
\begin{equation}\label{a_m,n=a_n/a_m c_n...:in N}
   a_{m,n} = \dfrac{a_n}{a_m} c_n
      \ \ \ \text{ and } \ \ \
      b_{m,n} = \dfrac{b_n}{b_m} c_n,
\end{equation}
where $\prts{a_n}_{n \in \N}$, $\prts{b_n}_{n \in \N}$ and $\prts{c_n}_{n \in \N}$ are sequences of positive numbers.

\begin{corollary}\label{cor:robustness:N:a_m,n=a_n/b_n c_n}
   Suppose that equation~\eqref{eq:x_m+1=A_m x_m, m in N} admits a general dichotomy with bounds given~\eqref{a_m,n=a_n/a_m c_n...:in N} and
   \begin{equation}\label{hyp:a_m b_n/(a_n b_m)<=C:in_N}
      \dfrac{a_m b_n}{a_n b_m} \le C
         \text{ for every } \prts{m,n} \in \N^2_\le
   \end{equation}
   for some number $C$. Let $B_n \colon X \to X$, $n \in \N$, be bounded linear operators such that
      $$ \|B_n\| \le \dfrac{a_n}{a_{n+1} c_{n+1}} \beta_n$$
   where $\beta:=\dsum_{k=1}^{+\infty} \beta_k < +\infty$. If $\beta C < 1$, then equation~\eqref{eq:x_m+1=(A_m+B_m)x_m, m in N} admits a dichotomy with bounds obtained by multiplying the bounds in~\eqref{a_m,n=a_n/a_m c_n...:in N} by $1/(1-\beta C)$.
\end{corollary}

Since the proof of Corollary~\ref{cor:robustness:N:a_m,n=a_n/b_n c_n} is similar to the proof of Corollary~\ref{cor:robustness:Z:a_m,n=a_n/b_n c_n}, we omit it.

Now we will apply the last corollary to the $\prts{\mu,\nu}-$dichotomies that were introduced in~\cite{Bento-Silva-JDDE} and \cite{Bento-Silva-NATMA-2012}, respectively, for differential equations and for difference equations.

\begin{corollary}
   Suppose that equation~\eqref{eq:x_m+1=A_m x_m, m in N} has a nonuniform $(\mu,\nu)-$dichotomy, i.e, admits a general dichotomy with bounds given by
   \begin{equation}\label{a_m,n=D(mu_m/mu_n)^a nu_n^eps}
      a_{m,n} = D \pfrac{\mu_m}{\mu_n}^a \nu_n^\eps
      \ \ \ \text{ and } \ \ \
      b_{m,n} = D \pfrac{\mu_m}{\mu_n}^b \nu_n^\eps,
   \end{equation}
   where $\prts{\mu_n}_{n\in\N}$ and $\prts{\nu_n}_{n\in\N}$ are increasing sequences of positive numbers, $D \ge 1$,  $a,b \in \R$ and $\eps \ge 0$. Consider bounded linear perturbations $B_n \colon X \to X$, $n \in \N$, such that
      $$ \|B_n\| \le \delta \nu_{n+1}^{-\gamma}$$
   where $\delta > 0$ and $\gamma \in \R$ is such that
      $  \kappa :=
         \dsum_{k=1}^{+\infty} \pfrac{\mu_k}{\mu_{k+1}}^a
            \nu_{k+1}^{\eps-\gamma} < \infty.$
   If
      $$ a \le b \ \ \ \text{ and } \ \ \ \theta:=D \delta \kappa < 1,$$
   then equation~\eqref{eq:x_m+1=(A_m+B_m)x_m, m in N} admits a nonuniform $\prts{\mu,\nu}-$dichotomy with $D$ replaced by $D/(1-\theta)$ in~\eqref{a_m,n=D(mu_m/mu_n)^a nu_n^eps}.
\end{corollary}

\begin{proof}
   Making in Corollary~\ref{cor:robustness:N:a_m,n=a_n/b_n c_n}
      $$ a_n = \mu_n^{-a}, \ \ \
         b_n = \mu_n^{-b}, \ \ \
         c_n = D \nu_n^{\eps} \ \ \ \text{ and } \ \ \
         \beta_n = D \delta \pfrac{\mu_n}{\mu_{n+1}}^a
            \nu_{n+1}^{\eps-\gamma},$$
   we have
      $$ \dfrac{a_n}{a_m} c_n
            = D \pfrac{\mu_m}{\mu_n}^a \nu_n^\eps, \ \ \
         \dfrac{b_n}{b_m} c_n
            = D \pfrac{\mu_m}{\mu_n}^b \nu_n^\eps \ \ \ \text{ and } \ \ \
         \dfrac{a_n}{a_{n+1} c_{n+1}} \beta_n
         = \delta \nu_{n+1}^{-\gamma}.$$
   Furthermore, since $a \le b$, for these bounds, inequality~\eqref{hyp:a_m b_n/(a_n b_m)<=C:in_N} is true with $C=1$. Hence we can apply Corollary~\ref{cor:robustness:N:a_m,n=a_n/b_n c_n} and taking into account that $\beta = D \delta \kappa$, the result follows immediately.
\end{proof}

The last corollary improves the robustness theorem proved by Chu in \cite[Theorem 3.11]{Chu-BSM-2013}. In fact, in \cite{Chu-BSM-2013} the nonuniform bounds of the dichotomy of the linear perturbation have exponent $2 \eps$ instead of $\eps$, the result is proved only for invertible cocycles and it was necessary to impose the additional conditions
   $$ \rho_1 := \dsum_{k=1}^{+\infty} \nu_{k+1}^{\eps-\gamma} < +\infty, \ \ \
      \rho_2 := \dsum_{k=1}^{+\infty} \nu_{k+1}^{2\eps-\gamma} < +\infty $$
   $$ a=b < 0, \ \ \
      D\delta(\kappa+\rho_1) < 1 \ \ \ \text{ and } \ \ \
      4 D \tilde D \delta \rho_2 <1,$$
where
   $$ \tilde D
      = \maxs{\dfrac{D}{1-D \delta (\kappa+\rho_1)},
         \dfrac{D}{1-D\delta\rho_1}}.$$

Now we will apply our result to nonuniform exponential dichotomies.

\begin{corollary} \label{cor:robustness:N:a_m,n=De^(a(m-n)+eps n)...}
   Suppose that equation~\eqref{eq:x_m+1=A_m x_m, m in N} has a nonuniform exponential dichotomy, i.e, admits a general dichotomy with bounds given by
   \begin{equation}\label{a_m,n=e^(a(m-n)+eps n)...}
      a_{m,n} = D\e^{a(m-n)+\eps n}
      \ \ \ \text{ and } \ \ \
      b_{m,n} = D \e^{b(m-n) + \eps n},
   \end{equation}
   with $D \ge 1$, $\eps \ge 0$ and $a,b \in \R$. Let $B_n \colon X \to X$, $n \in \N$, such that
      $$ \|B_n\| \le \delta \e^{-\gamma(n+1)}$$
   for some $\delta > 0$ and some $\gamma \in \R$. If
      $$ a \le b, \ \ \
         \eps - \gamma < 0 \ \ \ \text{ and } \ \ \
         \theta:= D  \delta \e^{2\eps-2\gamma-a} \prts{1-\e^{\eps-\gamma}}^{-1}
            < 1,$$
   then equation~\eqref{eq:x_m+1=(A_m+B_m)x_m, m in N} admits a nonuniform exponential dichotomy with bounds of the form~\eqref{a_m,n=e^(a(m-n)+eps n)...} with $D$ replaced by $D/(1-\theta)$.
\end{corollary}

\begin{proof}
   The result follows immediately from Corollary~\ref{cor:robustness:N:a_m,n=a_n/b_n c_n} with
      $$ a_n = \e^{-an}, \ \ \
         b_n = \e^{-bn}, \ \ \
         c_n = D \e^{\eps n} \ \ \ \text{ and } \ \ \
         \beta_n = D \delta \e^{-a} \e^{(\eps-\gamma)(n+1)}.$$
\end{proof}

Corollary~\ref{cor:robustness:N:a_m,n=De^(a(m-n)+eps n)...} improves Theorem 6 of Barreira, Silva and Valls~\cite{Barreira-Silva-Valls-JDE-2009} because in that paper the result is only for invertible cocycles, it was necessary to impose the three additional hypotheses
   $$ a=-b <0, \ \ \
      \eps < -a
      \ \ \ \text{ and } \ \ \
      \gamma = -2 \eps$$
and in the nonuniform bounds of the dichotomy of the perturbed equation the exponent $\eps$ was replaced by $2\eps$ and in the uniform bounds of the perturbed equation the exponent $a$ was replaced by $\tilde a > a$.

Now we will apply our result to improve the robustness of the dichotomies considered by Barreira, Fan, Valls and Zhang~\cite{Barreira-Fan-Valls-Zhang-TMNA-2011}, namely the nonuniform polynomial dichotomies with bounds of the form
\begin{equation}\label{a_m,n=D(m/n)^a n^eps...}
   a_{m,n} = D \pfrac{m}{n}^a n^\eps
   \ \ \ \text{ and } \ \ \
   b_{m,n} = D \pfrac{m}{n}^b n^\eps,
\end{equation}
with $D \ge 1$,  $a,b \in \R$ and $\eps \ge 0$.

\begin{corollary}\label{cor:robustness:N:a_m,n=D(m/n)^a n^eps...}
   Suppose that equation~\eqref{eq:x_m+1=A_m x_m, m in N} has a general dichotomy with bounds given by~\eqref{a_m,n=D(m/n)^a n^eps...}. Consider bounded linear perturbations $B_n \colon X \to X$, $n \in \N$, such that
      $$ \|B_n\| \le \delta \prts{n+1}^{-\gamma}$$
   for some $\delta >0$ and some $\gamma \in \R$. If
      $$ a \le b, \ \ \
         \eps - \gamma<-1 \ \ \ \text{ and } \ \ \
         \theta:= D \delta \maxs{1,2^{-a}} \prts{\zeta_{\gamma-\eps}-1} < 1,$$
   where $\zeta_{\gamma-\eps} = \dsum_{k=1}^{+\infty} k^{\eps-\gamma}$,
   then equation~\eqref{eq:x_m+1=(A_m+B_m)x_m, m in Z} admits a nonuniform polynomial dichotomy with bounds given by \eqref{a_m,n=D(m/n)^a n^eps...} with $D$ replaced $D/(1-\theta)$.
\end{corollary}

\begin{proof}
   The result follows immediately from Corollary~\ref{cor:robustness:N:a_m,n=a_n/b_n c_n} with
      $$ a_n = n^{-a}, \ \ \
         b_n = n^{-b}, \ \ \
         c_n = Dn^\eps \ \ \ \text{ and } \ \ \
         \beta_n = D \delta \pfrac{n}{n+1}^a (n+1)^{\eps-\gamma}.$$

\end{proof}

In~\cite{Barreira-Fan-Valls-Zhang-TMNA-2011}, Barreira, Fan, Valls and Zhang proved a robustness result for invertible cocycles with a nonuniform dichtomy with bounds of the form~\eqref{a_m,n=D(m/n)^a n^eps...} assuming that
   $$ \gamma > 2 \eps +1 \ \ \ \text{ and } \ \ \ \mins{-a,b} > \eps$$
and obtained a dichotomy for the perturbed equation with the exponent $\eps$ in the nonuniform bound replaced by $2 \eps$. Thus Corollary~\ref{cor:robustness:N:a_m,n=D(m/n)^a n^eps...} improves Theorem 3.1 in~\cite{Barreira-Fan-Valls-Zhang-TMNA-2011}.

Finally we apply our result to the nonuniform polynomial dichotomies introduced in~\cite{Bento-Silva-JFA-2009}.

\begin{corollary}
   Suppose that equation~\eqref{eq:x_m+1=A_m x_m, m in N} has a nonuniform polynomial dichotomy with bounds given by
   \begin{equation}\label{a_m,n=D(m-n+1)^a n^eps}
      a_{m,n} = D \prts{m-n+1}^a n^\eps
      \ \ \ \text{ and } \ \ \
      b_{m,n} = D \prts{n-m+1}^b n^\eps,
   \end{equation}
   where $D \ge 1$,  $a,b \in \R$ and $\eps \ge 0$. Let $B_n \colon X \to X$, $n \in \N$, be bounded linear perturbations such that
      $$ \|B_n\| \le \delta \prts{n+1}^{-\gamma}$$
   for some $\delta >0$ and $\gamma \in \R$. If
      $$ a \le b \le 0, \ \ \
         \gamma+\eps<-1 \ \ \ \text{ and } \ \ \
         \theta:= 2^{-a} D \delta \prts{\zeta_{\gamma-\eps}-1} < 1,$$
   where $\zeta_{\gamma-\eps} = \dsum_{k=1}^{+\infty} k^{\eps-\gamma}$, then equation~\eqref{eq:x_m+1=(A_m+B_m)x_m, m in Z} admits a nonuniform polynomial dichotomy with bounds given by~\eqref{a_m,n=D(m-n+1)^a n^eps} with $D$ replaced by $D/(1-\theta)$.
\end{corollary}

Since the proof of this corollary is analogous to the proof of Corollary~\ref{cor:robustness:Z:a_m,n=D(m-n+1)^a (|n|+1)^eps...} we omit it.

\section{Proof of Theorem~\ref{thm:robustness:Z}} \label{section:proof:robustness:Z}

The proof of Theorem~\ref{thm:robustness:Z} uses the Banach fixed point theorem in some suitable complete metric spaces.

For every $n \in \Z$, let $\Omg^+_n$ be the space of sequences $W = \prts{W_{m,n}}_{m \ge n} \subseteq B(X)$ such that
\begin{equation}\label{norm:Omg^+_n}
   \|W\|^+_n
   = \sup_{m \ge n} \dfrac{\|W_{m,n}\|}{a_{m,n}}
   < +\infty
\end{equation}
and let $\Omg^-_n$ be the space of sequences $Z = \prts{Z_{m,n}}_{m \le n} \subseteq B(X)$ such that
\begin{equation}\label{norm:Omg^-_n}
   \|Z\|^-_n
   = \sup_{m \le n} \dfrac{\|Z_{m,n}\|}{b_{m,n}}
   < +\infty.
\end{equation}
It is easy to prove that $\prts{\Omg^+_n, \|\cdot\|^+_n}$ and $\prts{\Omg^-_n, \|\cdot\|^-_n}$ are Banach spaces. We are also going to consider the Banach space $\Omg_n = \Omg^+_n \times \Omg^-_n$ equipped with the norm
\begin{equation}\label{norm:Omg_n}
   \|(W,Z)\|_n
   = \maxs{\|W\|_n^+, \|Z\|^-_n}.
\end{equation}

\begin{lemma}\label{lemma:J_n}
   Let $n \in \Z$. For every $\prts{W,Z} \in \Omg_n$, define
   \begin{equation}\label{def:J_n(W,Z)}
      J_n(W,Z) = \prts{J_{m,n}(W,Z)}_{m \ge n},
   \end{equation}
   where
   \begin{equation}\label{def:J_m,n(W,Z)}
      J_{m,n}(W,Z)
      = - \dsum_{k=-\infty}^{n-1} C_{m,k} Z_{k,n}
         + \dsum_{k=n}^{m-1} C_{m,k} W_{k,n}
         - \dsum_{k=m}^{+\infty} D_{m,k} W_{k,n}
   \end{equation}
   and
   \begin{equation}\label{def:C_m,k&D_m,k}
      C_{m,k} = \cA_{m,k+1} P_{k+1} B_k
      \ \ \ \text{ and } \ \ \
      D_{m,k} = \cA_{m,k+1} Q_{k+1} B_k.
   \end{equation}
   Then $J_n$ is a bounded linear operator from $\Omg_n$ into $\Omg_n^+$ and
   \begin{equation}\label{ine:||J_n||<=lbd}
      \|J_n\| \le \lbd,
   \end{equation}
   where $\lbd$ is given by~\eqref{def:lbd+mu}.
\end{lemma}

\begin{proof}
   It is obvious that from~\eqref{def:C_m,k&D_m,k} and~\ref{eq:(D1)} we obtain
   \begin{equation}\label{ine:||C_m,k||}
      \|C_{m,k}\|
      \le \|\cA_{m,k+1}P_{k+1}\| \ \|B_k\|
      \le a_{m,k+1} \|B_k\|
      \text{\ \ for all } \prts{m,k} \in \Z^2_>,
   \end{equation}
   and from~\eqref{def:C_m,k&D_m,k} and~\ref{eq:(D2)} we have
   \begin{equation}\label{ine:||D_m,k||}
      \|D_{m,k}\|
      \le \|\cA_{m,k+1}Q_{k+1}\| \ \|B_k\|
      \le b_{m,k+1} \|B_k\|
      \text{\ \ for all } \prts{m,k}\in \Z^2_\le.
   \end{equation}
   Moreover, by~\eqref{norm:Omg^+_n} and~\eqref{norm:Omg_n} it follows that
   \begin{equation}\label{ine:||W_k,n||<=a_k,n ||(W,Z)||_n}
      \|W_{k,n} \|
      \le a_{k,n} \|W\|^+_n
      \le a_{k,n} \|(W,Z)\|_n
      \text{\ \ for all } \prts{k,n} \in \Z^2_\ge
   \end{equation}
   and by~\eqref{norm:Omg^-_n} and~\eqref{norm:Omg_n} we get
   \begin{equation}\label{ine:||Z_k,n||<=b_k,n ||(W,Z)||_n}
      \|Z_{k,n}\|
      \le b_{k,n} \|Z\|^-_n
      \le b_{k,n} \|(W,Z)\|_n
      \text{\ \ for all } \prts{k,n} \in \Z^2_\le.
   \end{equation}
   Hence, from~\eqref{def:J_m,n(W,Z)}, \eqref{ine:||C_m,k||}, \eqref{ine:||D_m,k||}, \eqref{ine:||W_k,n||<=a_k,n ||(W,Z)||_n}, \eqref{ine:||Z_k,n||<=b_k,n ||(W,Z)||_n}, ~\eqref{def:lbd_m,n} and~\eqref{def:lbd+mu} we conclude that
   \begin{equation*}
      \begin{split}
         & \|J_{m,n}(W,Z)\|\\
         & \le \dsum_{k=-\infty}^{n-1} \|C_{m,k}\| \, \|Z_{k,n}\|
            + \dsum_{k=n}^{m-1} \|C_{m,k}\| \, \|W_{k,n}\|
            + \dsum_{k=m}^{+\infty} \|D_{m,k}\| \, \|W_{k,n}\|\\
         & \le \lbd_{m,n} \|(W,Z)\|_n\\
         & \le \lbd \, a_{m,n} \|(W,Z)\|_n.
      \end{split}
   \end{equation*}
   Hence $J_n(W,Z) \in \Omg_n^+$ and
   \begin{equation*}
      \|J_n(W,Z)\|_n^+ \le \lbd \, \|(W,Z)\|_n,
   \end{equation*}
   and this proves that  $J_n$ is a bounded linear operator from $\Omg_n$ into $\Omg^+_n$ and verifies~\eqref{ine:||J_n||<=lbd}.
\end{proof}

\begin{lemma}\label{lemma:L_n}
   Let $n \in \Z$. For every $\prts{W,Z} \in \Omg_n$, define
   \begin{equation}\label{def:L_n(W,Z)}
      L_n(W,Z) = \prts{L_{m,n}(W,Z)}_{m \le n},
   \end{equation}
   where
   \begin{equation}\label{def:L_m,n(W,Z)}
      L_{m,n}(W,Z)
      = \dsum_{k=-\infty}^{m-1} C_{m,k} Z_{k,n}
         - \dsum_{k=m}^{n-1} D_{m,k} Z_{k,n}
         +\dsum_{k=n}^{+\infty} D_{m,k} W_{k,n}
   \end{equation}
   $C_{m,k}$ and $D_{m,k}$ are defined by~\eqref{def:C_m,k&D_m,k}. Then $L_n$ is a bounded linear operator from $\Omg_n$ into $\Omg_n^-$ and
   \begin{equation}\label{ine:||L_n||<=mu}
      \|L_n\| \le \mu,
   \end{equation}
   where $\mu$ is given by~\eqref{def:lbd+mu}.
\end{lemma}

\begin{proof}
   From~\eqref{def:L_m,n(W,Z)}, \eqref{ine:||C_m,k||}, \eqref{ine:||D_m,k||}, \eqref{ine:||W_k,n||<=a_k,n ||(W,Z)||_n}, \eqref{ine:||Z_k,n||<=b_k,n ||(W,Z)||_n}, \eqref{def:mu_m,n} and~\eqref{def:lbd+mu} we have
   \begin{equation*}
      \begin{split}
         & \|L_{m,n}(W,Z)\|\\
         & \le \dsum_{k=-\infty}^{m-1} \|C_{m,k}\| \, \|Z_{k,n}\|
            + \dsum_{k=m}^{n-1} \|D_{m,k}\| \, \|Z_{k,n}\|
            + \dsum_{k=n}^{+\infty} \|D_{m,k}\| \, \|W_{k,n}\|\\
         & \le \mu_{m,n} \|(W,Z)\|_n\\
         & \le \mu \, b_{m,n} \|(W,Z)\|_n.
      \end{split}
   \end{equation*}
   Thus $L_n(W,Z) \in \Omg_n^-$,
   \begin{equation*}
      \|L_n(W,Z)\|_n^- \le \mu \, \|(W,Z)\|_n
   \end{equation*}
   and  $L_n$ is a bounded linear operator from $\Omg_n$ into $\Omg^-_n$ that verifies~\eqref{ine:||L_n||<=mu}.
\end{proof}

\begin{lemma}
   Let $n \in \Z$. For every $\prts{W,Z} \in \Omg_n$, let
   \begin{equation*}
      T_n(W,Z) = \prts{J_n(W,Z),L_n(W,Z)}
   \end{equation*}
   with $J_n$ defined by~\eqref{def:J_n(W,Z)} and $L_n$ defined by~\eqref{def:L_n(W,Z)}. Then $T_n$ is a bounded linear operator from $\Omg_n$ into $\Omg_n$ and
   \begin{equation}\label{norm:T_n}
      \|T_n\| \le \maxs{\lbd,\mu}.
   \end{equation}
\end{lemma}

\begin{proof}
   It follows immediately from Lemmas~\ref{lemma:J_n} and~\ref{lemma:L_n}.
\end{proof}

\begin{lemma} \label{lemma:U_m,n_&_V_m,n}
   Let $n \in \Z$. If~\eqref{ine:max(lbd,mu)<1} is satisfied, then there is a unique sequence $U=\prts{U_{m,n}}_{m \ge n} \in \Omg^+_n$  and a unique sequence $V=\prts{V_{m,n}}_{m \le n} \in \Omg^-_n$ such that
   \begin{equation}\label{def:U_m,n}
      \begin{split}
         U_{m,n}
         & = \cA_{m,n} P_n + J_{m,n}(U,V)\\
         & = \cA_{m,n} P_n
            - \dsum_{k=-\infty}^{n-1} C_{m,k} V_{k,n}
            + \dsum_{k=n}^{m-1} C_{m,k} U_{k,n}
            - \dsum_{k=m}^{+\infty} D_{m,k} U_{k,n}
      \end{split}
   \end{equation}
   for all $m \ge n$ and
   \begin{equation}\label{def:V_m,n}
      \begin{split}
         V_{m,n}
         & = \cA_{m,n} Q_n + L_{m,n}(U,V)\\
         & = \cA_{m,n} Q_n
            + \dsum_{k=-\infty}^{m-1} C_{m,k} V_{k,n}
            - \dsum_{k=m}^{n-1} D_{m,k} V_{k,n}
            + \dsum_{k=n}^{+\infty} D_{m,k} U_{k,n}.
      \end{split}
   \end{equation}
   for all $m \le n$. Moreover,
   \begin{equation}\label{ine:||U_m,n||}
      \|U_{m,n}\| \le \sigma \, a_{m,n} \text{ for every } (m,n) \in \Z^2_\ge
   \end{equation}
   and
   \begin{equation}\label{ine:||V_m,n||}
      \|V_{m,n}\| \le \sigma \, b_{m,n} \text{ for every } (m,n) \in \Z^2_\le.
   \end{equation}
\end{lemma}

\begin{proof}
   Let $n \in \Z$. Defining
      $$ \Gamma_n = \prts{\prts{\cA_{m,n} P_n}_{m \ge n},
         \prts{\cA_{m,n} Q_n}_{m \le n}},$$
   it follows from~\ref{eq:(D1)} and~\ref{eq:(D2)} that $\Gamma_n \in \Omg_n$. From~\eqref{norm:T_n} and~\eqref{ine:max(lbd,mu)<1} the operator $\Upsilon_n \colon \Omg_n \to \Omg_n$ defined by
      $$ \Upsilon_n = \Gamma_n + T_n$$
   is a contraction with Lipschitz constant less or equal than $\maxs{\lbd,\mu}<1$. Since $\Omg_n$ is a Banach space, the Banach fixed point Theorem assure us that $\Upsilon_n$ has a unique fixed point $\prts{U,V} \in \Omg_n$. Obviously, from the definition of $\Upsilon_n$, the fixed point $\prts{U,V}$ verifies~\eqref{def:U_m,n} and~\eqref{def:V_m,n}.

   Also from the Banach fixed point Theorem we have
      $$ \|\prts{U,V}- 0\|_n
         \le \sigma \|T_n 0 -0\|_n
         = \sigma \|\Gamma_n\|_n.$$
   Since $\|\Gamma_n\|_n \le 1$ we have $\|\prts{U,V}\|_n \le \sigma$
   and thus ~\eqref{ine:||U_m,n||} and~\eqref{ine:||V_m,n||} are satisfied.
\end{proof}

\begin{lemma}\label{lemma:U_m,n+Vm,n_solutions}
   The sequences $\prts{U_{m,n}}_{m \ge n}$ and $\prts{V_{m,n}}_{m \le n}$ are solutions of equation~\eqref{eq:x_m+1=(A_m+B_m)x_m, m in Z}.
\end{lemma}

\begin{proof}
   This is an immediate consequence of
   \begin{align*}
      & A_m U_{m,n}\\
      & = A_m \prts{\cA_{m,n} P_n - \dsum_{k=-\infty}^{n-1} C_{m,k} V_{k,n}
         + \dsum_{k=n}^{m-1} C_{m,k} U_{k,n} - \dsum_{k=m}^{+\infty} D_{m,k} U_{k,n}}\\
      & = \cA_{m+1,n} P_n - \dsum_{k=-\infty}^{n-1} C_{m+1,k} V_{k,n}
         + \dsum_{k=n}^{m-1} C_{m+1,k} U_{k,n} - \dsum_{k=m}^{+\infty} D_{m+1,k} U_{k,n}\\
      & = \cA_{m+1,n} P_n - \dsum_{k=-\infty}^{n-1} C_{m+1,k} V_{k,n}
         + \dsum_{k=n}^{m} C_{m+1,k} U_{k,n} - \dsum_{k=m+1}^{+\infty} D_{m+1,k} U_{k,n}\\
      &  \ \ \ - C_{m+1,m} U_{m,n} - D_{m+1,m} U_{m,n}\\
      & = U_{m+1,n} - B_m U_{m,n}
   \end{align*}
   and
   \begin{align*}
      & A_m V_{m,n}\\
      & = A_m \prts{\cA_{m,n} Q_n + \dsum_{k=-\infty}^{m-1} C_{m,k} V_{k,n}
          - \dsum_{k=m}^{n-1} D_{m,k} V_{k,n} + \dsum_{k=n}^{+\infty} D_{m,k} U_{k,n}}\\
      & = \cA_{m+1,n} Q_n + \dsum_{k=-\infty}^{m-1} C_{m+1,k} V_{k,n}
          - \dsum_{k=m}^{n-1} D_{m+1,k} V_{k,n} + \dsum_{k=n}^{+\infty} D_{m+1,k} U_{k,n}\\
      & = \cA_{m+1,n} Q_n + \dsum_{k=-\infty}^{m} C_{m+1,k} V_{k,n}
          - \dsum_{k=m+1}^{n-1} D_{m+1,k} V_{k,n} + \dsum_{k=n}^{+\infty} D_{m+1,k} U_{k,n}\\
      &  \ \ \ - C_{m+1,m} V_{m,n} - D_{m+1,m} V_{m,n}\\
      & = V_{m+1,n} - B_m V_{m,n}.
   \end{align*}
\end{proof}

\begin{lemma}\label{lemma:Um,jUj,n=Um,n}
   Let  $\prts{j,n} \in \Z^2_\ge$. Then
      $$ U_{m,j}U_{j,n} = U_{m,n} \text{ for every $m \ge j$}$$
   and
      $$ V_{m,j}U_{j,n} = 0 \text{ for every $m \le j$}.$$
\end{lemma}

\begin{proof}
   Let $\prts{j, n} \in \Z^2_\ge$ and define
      $$ W_{m,j} = U_{m,j} U_{j,n} - U_{m,n} \text{ for every } m \ge j$$
   and
      $$ Z_{m,j} = V_{m,j} U_{j,n} \text{ for every } m \le j.$$
   From
   \begin{align*}
      \norm{W_{m,j}}
      & \le \|U_{m,j}\| \ \|U_{j,n}\| + \|U_{m,n}\|\\
      & \le \sigma a_{m,j} \|U_{j,n}\| + \sigma a_{m,n}\\
      & \le \sigma a_{m,j} \prts{\|U_{j,n}\| + a_{m,n}/a_{m,j}}
   \end{align*}
   and because we are assuming~\eqref{hyp:sup_m_a_m,n/a_m,j<infty} we have $W = \prts{W_{m,j}}_{m \ge j} \in \Omg_j^+$. On the other hand, since
      $$ \norm{Z_{m,j}}
         \le \|V_{m,j}\| \ \|U_{j,n}\|\\
         \le \sigma b_{m,j} \|U_{j,n}\|,$$
   it follows that $Z = \prts{Z_{m,j}}_{m \le j} \in \Omg^-_j$. Hence $\prts{W,Z} \in \Omg_j$.

   Now we will prove that $(W,Z)$ is a fixed point of $T_j$. For $m \ge j$, from~\eqref{def:U_m,n} we obtain
   \begin{equation}\label{eq:W_m,j=U_m,j U_j,n - U_m,n=...}
      W_{m,j}
      = U_{m,j}U_{j,n} - U_{m,n}
      = \cA_{m,j} P_j U_{j,n} + J_{m,j}(U,V)U_{j,n} - U_{m,n}.
   \end{equation}
   Since from~\eqref{def:U_m,n} we have
   \begin{equation}\label{eq:A_m,j P_j U_j,n = ...}
      \cA_{m,j} P_j U_{j,n}
      = \cA_{m,n}P_n
         - \dsum_{k=- \infty}^{n-1} C_{m,k} V_{k,n}
         + \dsum_{k=n}^{j-1} C_{m,k} U_{k,n},
   \end{equation}
   and from~\eqref{def:J_m,n(W,Z)} we have
   \begin{equation}\label{eq:J_m,j(U,V) U_j,n = ...}
      \begin{split}
         & J_{m,j}(U,V) U_{j,n}\\
         & =  - \dsum_{k=-\infty}^{j-1} C_{m,k}V_{k,j}U_{j,n}
            + \dsum_{k=j}^{m-1} C_{m,k}U_{k,j}U_{j,n}
            - \dsum_{k=m}^{+\infty} D_{m,k}U_{k,j}U_{j,n},
      \end{split}
   \end{equation}
   it follows from~\eqref{eq:W_m,j=U_m,j U_j,n - U_m,n=...}, \eqref{eq:A_m,j P_j U_j,n = ...}, \eqref{eq:J_m,j(U,V) U_j,n = ...} and~\eqref{def:U_m,n} that
   \begin{align*}
      W_{m,j}
      & = - \dsum_{k=-\infty}^{j-1} C_{m,k}V_{k,j}U_{j,n}
            + \dsum_{k=j}^{m-1} C_{m,k}\prts{U_{k,j}U_{j,n} - U_{k,n}}\\
      & \ \ \ \
         - \dsum_{k=m}^{+\infty} D_{m,k}\prts{U_{k,j}U_{j,n} - U_{k,n}}\\
      & = J_{m,n}(W,Z).
   \end{align*}

   For $m \le j$ by~\eqref{def:V_m,n} we have
   \begin{equation}\label{eq:Z_m,j = V_m,j U_j,n}
      Z_{m,j}
      = V_{m,j}U_{j,n}
      = \cA_{m,j}Q_j U_{j,n} + L_{m,j}(U,V) U_{j,n},
   \end{equation}
   by~\eqref{def:U_m,n} we obtain
   \begin{equation}\label{eq:A_m,j Q_j U_j,n = ...}
      \cA_{m,j}Q_j U_{j,n}
      = - \dsum_{k=j}^{+\infty} D_{m,k} U_{k,n}
   \end{equation}
   and from~\eqref{def:L_m,n(W,Z)} we get
   \begin{equation}\label{eq:L_m,j(U,V) U_j,n = ...}
      L_{m,j}(U,V) U_{j,n}
      = \dsum_{k=-\infty}^{m-1} C_{m,k} V_{k,j}U_{j,n}
         - \dsum_{k=m}^{j-1} D_{m,k} V_{k,j}U_{j,n}
         + \dsum_{k=j}^{+\infty} D_{m,k} U_{k,j} U_{j,n}.
   \end{equation}
   Hence, it follows from~\eqref{eq:Z_m,j = V_m,j U_j,n}, \eqref{eq:A_m,j Q_j U_j,n = ...} and~\eqref{eq:L_m,j(U,V) U_j,n = ...} that
   \begin{align*}
      Z_{m,j}
      & = \dsum_{k=-\infty}^{m-1} C_{m,k} V_{k,j}U_{j,n}
         - \dsum_{k=m}^{j-1} D_{m,k} V_{k,j}U_{j,n}
         + \dsum_{k=j}^{+\infty} D_{m,k} \prts{U_{k,j} U_{j,n}-U_{k,n}}\\
      & = L_{m,j}(W,Z).
   \end{align*}

   Therefore $(W,Z)$ is a fixed point of the linear contraction $T_j$. Since $T_j$ has a unique fixed point, the zero of $\Omg_j$, we must have
      $$ W_{m,j} = U_{m,j} U_{j,n} - U_{m,n} = 0$$
   for all $m \ge j$ and
      $$ Z_{m,j} = V_{m,j} U_{j,n} =0,$$
   for all $m \le j$.
\end{proof}

\begin{lemma}\label{lemma:Vm,jVj,n=Vm,n}
   Let $\prts{j,n} \in \Z^2_\le$. Then
      $$ U_{m,j}V_{j,n} = 0 \text{ for every $m \ge j$}$$
   and
      $$ V_{m,j}V_{j,n} = V_{m,n} \text{ for every $m \le j$}.$$
\end{lemma}

\begin{proof}
   Let $\prts{j,n} \in \Z^2_\le$ and consider
      $$ W_{m,j} = U_{m,j} V_{j,n}, \ m \ge j,$$
   and
      $$ Z_{m,j} = V_{m,j} V_{j,n} - V_{m,n} , \ m \le j.$$
   Since
      $$ \norm{W_{m,j}}
         \le \|U_{m,j}\| \ \|V_{j,n}\|
         \le \sigma a_{m,j} \|V_{j,n}\|,$$
   it follows that $W = \prts{W_{m,j}}_{m \ge j} \in \Omg_j^+$. On the other hand, we have
   \begin{align*}
      \norm{Z_{m,j}}
      & \le \|V_{m,j}\| \ \|V_{j,n}\| + \|V_{m,n}\|\\
      & \le \sigma b_{m,j} \|U_{j,n}\| + \sigma b_{m,n} \\
      & \le \sigma b_{m,j} \prts{\|U_{j,n}\|  + b_{m,n}/b_{m,j}},
   \end{align*}
   and from~\eqref{hyp:sup_m_b_m,n/b_m,j<infty} it follows that $Z = \prts{W_{m,j}}_{m \le j} \in \Omg^-_j$. Hence $\prts{W,Z} \in \Omg_j$.

   Now we will prove that $(W,Z)$ is a fixed point of $T_j$. For $m \ge j$ by~\eqref{def:U_m,n} we have
   \begin{equation}\label{eq:W_m,j=U_m,j V_j,n}
      W_{m,j}
      = U_{m,j}V_{j,n}
      = \cA_{m,j}P_j V_{j,n} + J_{m,j}(U,V) V_{j,n}
   \end{equation}
   and because from~\eqref{def:V_m,n} we have
   \begin{equation}\label{eq:A_m,j P_j V_{j,n} = ...}
      \cA_{m,j}P_j V_{j,n}
      = \dsum_{k=-\infty}^{j-1} C_{m,k} V_{k,n}
   \end{equation}
   and from~\eqref{def:J_m,n(W,Z)} we obtain
   \begin{equation}\label{eq:J_m,j(U,V) V_j,n = ...}
      \begin{split}
         & J_{m,j}(U,V) V_{j,n}\\
         & = - \dsum_{k=-\infty}^{j-1} C_{m,k} V_{k,j}V_{j,n}
            + \dsum_{k=j}^{m-1} C_{m,k} U_{k,j}V_{j,n}
            - \dsum_{k=j}^{+\infty} D_{m,k} U_{k,j} V_{j,n},
      \end{split}
   \end{equation}
   it follows from~\eqref{eq:W_m,j=U_m,j V_j,n}, \eqref{eq:A_m,j P_j V_{j,n} = ...} and~\eqref{eq:J_m,j(U,V) V_j,n = ...} that
   \begin{align*}
      W_{m,j}
      & = -\dsum_{k=-\infty}^{j-1} C_{m,k} \prts{V_{k,j}V_{j,n} - V_{k,n}}
         + \dsum_{k=j}^{m-1} D_{m,k} U_{k,j}V_{j,n}
         - \dsum_{k=j}^{+\infty} D_{m,k} U_{k,j} V_{j,n}\\
      & = J_{m,j}(W,Z).
   \end{align*}

   For $m \le j$ by~\eqref{def:V_m,n} we have
   \begin{equation}\label{eq:Z_m,j=V_m,jV_j,n-V_m,n}
      Z_{m,j}
      = V_{m,j}V_{j,n} - V_{m,n}
      = \cA_{m,j} Q_j V_{j,n} + L_{m,j}(U,V)V_{j,n} - V_{m,n}
   \end{equation}
   and
   \begin{equation}\label{eq:A_m,j Q_j V_j,n=...}
      \cA_{m,j} Q_j V_{j,n}
      = \cA_{m,n}Q_n
         - \dsum_{k=j}^{n-1} D_{m,k} V_{k,n}
         + \dsum_{k=n}^{+\infty} D_{m,k} U_{k,n}
   \end{equation}
   and from~\eqref{def:L_m,n(W,Z)} we get
   \begin{equation}\label{eq:L_m,j(U,V) V_j,n=...}
      L_{m,j}(U,V) V_{j,n}
      = \dsum_{k=-\infty}^{m-1} C_{m,k}V_{k,j}V_{j,n}
         - \dsum_{k=m}^{j-1} D_{m,k}V_{k,j}V_{j,n}
         + \dsum_{k=j}^{+\infty} D_{m,k}U_{k,j}V_{j,n}.
   \end{equation}
   Thus it follows from~\eqref{eq:Z_m,j=V_m,jV_j,n-V_m,n}, \eqref{eq:A_m,j Q_j V_j,n=...}, ~\eqref{eq:L_m,j(U,V) V_j,n=...} and~\eqref{def:V_m,n} that
   \begin{align*}
      Z_{m,j}
      & = \dsum_{k=-\infty}^{m-1} C_{m,k} \prts{V_{k,j}V_{j,n} - V_{k,n}}
            - \dsum_{k=m}^{j-1} D_{m,k} \prts{V_{k,j}V_{j,n} - V_{k,n}}+\\
      & \qquad + \dsum_{k=j}^{+\infty} D_{m,k} U_{k,j}V_{j,n}\\
      & = L_{m,j}(W,Z).
   \end{align*}

   Hence $(W,Z)$ is a fixed point of the linear contraction $T_j$. But, since the only fixed point of $T_j$ is zero, it follows that $W_{m,j} = U_{m,j} V_{j,n} = 0$ for every $m \ge j$ and $Z_{m,j} = V_{m,j} V_{j,n} -  V_{m,n}=0$ for every $m \le j$.
\end{proof}

Letting
   $$ \hcA_{m,n} =
      \begin{cases}
         \prts{A_{m-1}+B_{m-1}} \ldots \prts{A_n+B_n} & \text{ if } m > n,\\
         \Id & \text{ if } m = n,
      \end{cases}$$
it is clear that
\begin{equation}\label{eq:hAm,n=A_m,n+sum A_m,k B_k hA_k,n}
   \hcA_{m,n} =
   \cA_{m,n} + \dsum_{k=n}^{m-1} \cA_{m,k+1} B_k \hcA_{k,n}
\end{equation}
for every $\prts{m,n} \in \Z^2_\ge$.

\begin{lemma} \label{lemma:U_m,j_hcA_j,n_V_n,n_=_0}
   Let $\prts{j,n} \in \Z^2_\ge$. Then for every $m \ge j$
      $$ U_{m,j} \hcA_{j,n} V_{n,n} = 0$$
   and for every $m \le j$
      $$ V_{m,j} \hcA_{j,n} V_{n,n} =
         \begin{cases}
            \hcA_{m,n} V_{n,n}, & \text{ if } \ n \le m \le j,\\
            V_{m,n}, & \text{ if } \ m \le n.
         \end{cases}$$
\end{lemma}

\begin{proof}
   Let $\prts{j, n} \in \Z^2_\ge$ and define
      $$ W_{m,j} = U_{m,j} \hcA_{j,n} V_{n,n}, \ \ m \ge j,$$
   and
      $$ Z_{m,j} =
         \begin{cases}
            V_{m,j} \hcA_{j,n} V_{n,n} -  \hcA_{m,n} V_{n,n}
               & \text{ if } n \le m \le j,\\
            V_{m,j} \hcA_{j,n} V_{n,n} -  V_{m,n}
               & \text{ if } m \le n.\\
         \end{cases}.$$
   Putting $W = \prts{W_{m,j}}_{m \ge j}$ and $Z = \prts{Z_{m,j}}_{m \le j}$ it is clear that $\prts{W,Z} \in \Omg_j$.

   For every $m \ge j$, it follows from~\eqref{def:U_m,n} that
   \begin{equation}\label{eq:U_m,j hA_j,n V_n,n=...}
      U_{m,j} \hcA_{j,n} V_{n,n}
      = \cA_{m,j} P_j \hcA_{j,n} V_{n,n} + J_{m,j} (U,V) \hcA_{j,n} V_{n,n}.
   \end{equation}
   Because from~\eqref{eq:hAm,n=A_m,n+sum A_m,k B_k hA_k,n} and~\eqref{def:V_m,n} we have
   \begin{equation}\label{eq:A_m,j P_j hA_j,n V_n,n=...}
      \begin{split}
         \cA_{m,j} P_j \hcA_{j,n} V_{n,n}
         & = \cA_{m,n} P_n V_{n,n}
            + \dsum_{k=n}^{j-1} \cA_{m,k+1} P_{k+1} B_k \ \hcA_{k,n} V_{n,n}\\
         & = \dsum_{k=-\infty}^{n-1} C_{m,k} V_{k,n}
            + \dsum_{k=n}^{j-1} C_{m,k} \ \hcA_{k,n} V_{n,n},
      \end{split}
   \end{equation}
   and from~\eqref{def:J_m,n(W,Z)} we obtain
   \begin{equation}\label{eq:J_m,j(U,V) hA_j,n V_n,n=...}
      \begin{split}
         & J_{m,j} (U,V) \hcA_{j,n} V_{n,n}\\
         & = - \dsum_{k=-\infty}^{j-1} C_{m,k} V_{k,j} \hcA_{j,n} V_{n,n}
            + \dsum_{k=j}^{m-1} C_{m,k} U_{k,j} \hcA_{j,n} V_{n,n}
            - \dsum_{k=m}^{+\infty} D_{m,k} U_{k,j} \hcA_{j,n} V_{n,n},
      \end{split}
   \end{equation}
   by~\eqref{eq:U_m,j hA_j,n V_n,n=...}, \eqref{eq:A_m,j P_j hA_j,n V_n,n=...} and~\eqref{eq:J_m,j(U,V) hA_j,n V_n,n=...} we can conclude that
   \begin{align*}
      W_{m,j}
      & = U_{m,j} \hcA_{j,n} V_{n,n}\\
      & = - \dsum_{k=-\infty}^{j-1} C_{m,k} Z_{k,j}
         + \dsum_{k=j}^{m-1} C_{m,k} W_{k,j}
         - \dsum_{k=m}^{+\infty} D_{m,k} W_{k,j}\\
      & = J_{m,j}(W,Z).
   \end{align*}

   On the other hand, for $m \le j$, from~\eqref{def:V_m,n} we have
   \begin{equation} \label{eq:V_m,j_hcA_j,n_V_n,n}
      V_{m,j} \hcA_{j,n} V_{n,n}
      = A_{m,j} Q_j \hcA_{j,n} V_{n,n} + L_{m,j}(U,V) \hcA_{j,n} V_{n,n},
   \end{equation}
   from~\eqref{eq:hAm,n=A_m,n+sum A_m,k B_k hA_k,n} it follows that
   \begin{equation}\label{eq:A_m,j_Q_j_hcA_j,n_V_n,n}
      \begin{split}
         A_{m,j} Q_j \hcA_{j,n} V_{n,n}
         & = A_{m,n} Q_n V_{n,n}
            + \dsum_{k=n}^{j-1} \cA_{m,k+1} Q_{k+1} B_k \hcA_{k,n} V_{n,n}\\
      \end{split}
   \end{equation}
   and by~\eqref{def:L_m,n(W,Z)} we get
   \begin{equation}\label{eq:Lm,j(U,V) hA_j,n V_n,n=...}
      \begin{split}
         & L_{m,j}(U,V) \hcA_{j,n} V_{n,n}\\
         & = \dsum_{k=-\infty}^{m-1} C_{m,k} V_{k,j} \hcA_{j,n} V_{n,n}
            - \dsum_{k=m}^{j-1} D_{m,k} V_{k,j} \hcA_{j,n} V_{n,n}
            + \dsum_{k=j}^{+\infty} D_{m,k} U_{k,j} \hcA_{j,n} V_{n,n}.
      \end{split}
   \end{equation}

   If $n \le m \le j$, from~\eqref{eq:A_m,j_Q_j_hcA_j,n_V_n,n},~\eqref{eq:hAm,n=A_m,n+sum A_m,k B_k hA_k,n} and~\eqref{def:V_m,n} we have
   \begin{equation*}
      \begin{split}
         & \cA_{m,j} Q_j \hcA_{j,n} V_{n,n} - \hcA_{m,n} V_{n,n}\\
         & = \prts{A_{m,n} Q_n
            + \dsum_{k=n}^{j-1} \cA_{m,k+1} Q_{k+1} B_k \hcA_{k,n}
            - \cA_{m,n}
            - \sum_{k=n}^{m-1} \cA_{m,k+1} B_k \hcA_{k,n}} V_{n,n}\\
         & = \prts{-\cA_{m,n} P_n
               - \dsum_{k=n}^{m-1} \cA_{m,k+1} P_{k+1} B_k \hcA_{k,n}
               + \dsum_{k=m}^{j-1} \cA_{m,k+1} Q_{k+1} B_k \hcA_{k,n}} V_{n,n}\\
         & = -\cA_{m,n} P_n V_{n,n}
               - \dsum_{k=n}^{m-1} \cA_{m,k+1} P_{k+1} B_k \hcA_{k,n} V_{n,n}
               + \dsum_{k=m}^{j-1} \cA_{m,k+1} Q_{k+1} B_k \hcA_{k,n} V_{n,n}\\
         & = -\dsum_{k=- \infty}^{n-1} C_{m,k} V_{k,n}
            - \dsum_{k=n}^{m-1} C_{m,k} \hcA_{k,n} V_{n,n}
            + \dsum_{k=m}^{j-1} D_{m,k} \hcA_{k,n} V_{n,n},
      \end{split}
   \end{equation*}
   and this together with~\eqref{eq:V_m,j_hcA_j,n_V_n,n} and~\eqref{eq:Lm,j(U,V) hA_j,n V_n,n=...} imply
   \begin{align*}
      Z_{m,j}
      & = V_{m,j} \hcA_{j,n} V_{n,n} - \hcA_{m,n} V_{n,n}\\
      & = \dsum_{k=-\infty}^{m-1} C_{m,k} Z_{k,j}
         - \dsum_{k=m}^{j-1} D_{m,k} Z_{k,j}
         + \dsum_{k=j}^{+\infty} D_{m,k} W_{k,j}\\
      & = L_{m,j}(W,Z),
   \end{align*}
   for $n \le m \le j$.

   If $m \le n \le j$, by~\eqref{eq:V_m,j_hcA_j,n_V_n,n},~\eqref{eq:A_m,j_Q_j_hcA_j,n_V_n,n} and~\eqref{def:V_m,n} we have
   \begin{align*}
      \cA_{m,j} Q_j \hcA_{j,n} V_{n,n} - V_{m,n}
      & = A_{m,n} Q_n V_{n,n}
         + \dsum_{k=n}^{j-1} \cA_{m,k+1} Q_{k+1} B_k \hcA_{k,n} V_{n,n}
         - V_{m,n}\\
      & = A_{m,n} Q_n + \dsum_{k=n}^{+\infty} D_{m,k} U_{k,n}
            + \dsum_{k=n}^{j-1} D_{m,k} \hcA_{k,n} V_{n,n}
            - V_{m,n}\\
      & = \dsum_{k=n}^{j-1} D_{m,k} \hcA_{k,n} V_{n,n}
         - \dsum_{k=-\infty}^{m-1} C_{m,k} V_{k,n}
         - \dsum_{k=m}^{n-1} D_{m,k} V_{k,n}
   \end{align*}
   and this together with~\eqref{eq:V_m,j_hcA_j,n_V_n,n} implies
   \begin{align*}
      Z_{m,j}
      & = V_{m,j} \hcA_{j,n} V_{n,n} - V_{m,n}\\
      & = \dsum_{k=-\infty}^{m-1} C_{m,k} Z_{k,j}
         - \dsum_{k=m}^{j-1} D_{m,k} Z_{k,j}
         + \dsum_{k=j}^{+\infty} D_{m,k} W_{k,j}\\
      & = L_{m,j}(W,Z).
   \end{align*}

   Therefore, $(W,Z)$ is a fixed point of $T_j$ and this implies $W_{m,j} = 0$ for every $m \ge j$ and $Z_{m,j} = 0$ for every $m \le j$.
\end{proof}

Define $\widehat{E}_n = U_{n,n}(X)$ and $\widehat{F}_n = V_{n,n} (X)$.

\begin{lemma}\label{lemma:inclusions}
   For every $\prts{m,n} \in \Z^2_\ge$ we have
      $$ \hcA_{m,n}(\widehat E_n) \subseteq \widehat E_m
         \ \ \ \text{ and } \ \ \
         \hcA_{m,n}(\widehat F_n) \subseteq \widehat F_m.$$
\end{lemma}

\begin{proof}
   Let $x \in \widehat E_n$. Since $U_{n,n} U_{n,n} = U_{n,n}$, we have $x = U_{n,n} x$. Hence by Lemmas~\ref{lemma:U_m,n+Vm,n_solutions} and~\ref{lemma:Um,jUj,n=Um,n} we get
      $$ \hcA_{m,n} x
         = \hcA_{m,n} U_{n,n} x
         = U_{m,n} x
         = U_{m,m} U_{m,n} x \in \widehat E_m$$
   and this proves the first inclusion.

   Let $x \in \widehat F_n$. Since $x=V_{n,n}x$ we have from Lemma~\ref{lemma:U_m,j_hcA_j,n_V_n,n_=_0}
      $$ \hcA_{m,n} x
         = \hcA_{m,n} V_{n,n} x
         = V_{m,m} \hcA_{m,n} V_{n,n} x \in \widehat F_m$$
   and the second inclusion is proved.
\end{proof}

\begin{lemma}\label{lemma:invertibility}
   For every $\prts{m,n} \in \Z^2_\ge$, the restriction $\hcA_{m,n}|_{\widehat F_n} \colon  \widehat F_n \to \widehat F_m$ is invertible and its inverse is the restriction $V_{n,m}|_{\widehat F_m} \colon \widehat F_m \to \widehat F_n$.
\end{lemma}

\begin{proof}
   Since
      $$ \hcA_{m,n} V_{n,m} = V_{m,m}$$
   it follows that for every $x \in \widehat F_m$ we have
      $$ \hcA_{m,n} V_{n,m} x = V_{m,m} x = x$$
   and this prove that $\hcA_{m,n}|_{\widehat F_n} \colon  \widehat F_n \to \widehat F_m$ is surjective.

   To prove that $\hcA_{m,n}|_{\widehat F_n} \colon  \widehat F_n \to \widehat F_m$ is injective, we only have to show that if $\hcA_{m,n} y = 0$ for some $y \in \widehat F_n$, then $y=0$. Suppose that $y \in F_n$ and $\hcA_{m,n} y=0$. From~\eqref{eq:hAm,n=A_m,n+sum A_m,k B_k hA_k,n} we have
      $$ \cA_{m,n} y
         = - \dsum_{k = n}^{m-1} \cA_{m,k+1} B_k \hcA_{k,n} y$$
   and this implies
   \begin{equation}\label{eq:A_m,n Q_n y =...:hA_m,n y=0}
      \cA_{m,n} Q_n y
      = - \dsum_{k = n}^{m-1} \cA_{m,k+1} Q_{k+1} B_k \hcA_{k,n} y.
   \end{equation}
   Since $\cA_{m,n}$ is invertible in $F_n$, from~\eqref{eq:A_m,n Q_n y =...:hA_m,n y=0} we have
   \begin{equation}\label{eq:Qny=...:hA_m,ny=0}
      Q_n y
      = - \dsum_{k = n}^{m-1} \cA_{n,k+1} Q_{k+1} B_k \hcA_{k,n} y
      = - \dsum_{k = n}^{m-1} D_{n,k} \hcA_{k,n}y
   \end{equation}
   and it follows from~\eqref{def:V_m,n}, Lemma~\ref{lemma:Vm,jVj,n=Vm,n} and~\eqref{eq:Qny=...:hA_m,ny=0} that
   \begin{equation}\label{eq:y=...:A_m,n y=0}
      y
      = V_{n,n}y
      = Q_n y + \dsum_{k=-\infty}^{n-1} C_{n,k} V_{k,n} y
      = \dsum_{k=-\infty}^{n-1} C_{n,k} V_{k,n} y
      - \dsum_{k = n}^{m-1} D_{n,k} \hcA_{k,n} y.
   \end{equation}
   Let
      $$ \eta
         = \sups{\sup_{j \le n} \dfrac{\|V_{j,n}y\|}{b_{j,n}},
            \max_{n \le j < m} \dfrac{\|\hcA_{j,n}y\|}{a_{j,n}}}.$$
   From $\|V_{j,n}y\| \le \sigma b_{j,n} \|y\|$ for every $\prts{j,n} \in \Z^2_\le$, it follows that $\eta < +\infty$. By~\eqref{def:V_m,n}, Lemma \ref{lemma:Vm,jVj,n=Vm,n} and~\eqref{eq:Qny=...:hA_m,ny=0}, we have for every $j \le n$
   \begin{align*}
      V_{j,n} y
      & = A_{j,n} Q_n y + \dsum_{k=- \infty}^{j-1} C_{j,k} V_{k,n} y
         - \dsum_{k=j}^{n-1} D_{j,k} V_{k,n} y,\\
      & = \dsum_{k=- \infty}^{j-1} C_{j,k} V_{k,n} y
         - \dsum_{k=j}^{n-1} D_{j,k} V_{k,n} y
         - \dsum_{k=n}^{m-1} D_{j,k} \hcA_{k,n} y
   \end{align*}
   and this implies
   \begin{equation}\label{ine:||Vj,n_y||:hA_m,n y=0}
      \begin{split}
         \|V_{j,n} y\|
         & \le \dsum_{k=- \infty}^{j-1} \|C_{j,k}\| b_{k,n} \eta
            + \dsum_{k=j}^{n-1} \|D_{j,k}\| b_{k,n} \eta
            + \dsum_{k=n}^{m-1} \|D_{j,k}\| a_{k,n} \eta\\
         & \le \mu_{j,n} \eta\\
         & \le \mu \, b_{j,n} \eta.
      \end{split}
   \end{equation}
   On the other hand, for $n \le j < m$, from~\eqref{eq:hAm,n=A_m,n+sum A_m,k B_k hA_k,n} and \eqref{eq:y=...:A_m,n y=0} we have
   \begin{align*}
      \hcA_{j,n} y
      & = \cA_{j,n} y + \dsum_{k=n}^{j-1} \cA_{j,k+1} B_k \hcA_{k,n}y\\
      & = \dsum_{k=-\infty}^{n-1} C_{j,k} V_{k,n} y
         - \dsum_{k=n}^{m-1} D_{j,k} \hcA_{k,n}y
         + \dsum_{k=n}^{j-1} \cA_{j,k+1} B_k \hcA_{k,n}y\\
      & = \dsum_{k=-\infty}^{n-1} C_{j,k} V_{k,n} y
         + \dsum_{k=n}^{j-1} C_{j,k} \hcA_{k,n}y
         - \dsum_{k=j}^{m-1} D_{j,k} \hcA_{k,n}y,
   \end{align*}
   and thus
   \begin{equation}\label{ine:||hAj,n_y||:A_m,n y=0}
      \begin{split}
         \|\hcA_{j,n} y\|
         & \le \dsum_{k=-\infty}^{n-1} \|C_{j,k}\| b_{k,n} \eta
            + \dsum_{k=n}^{j-1} \|C_{j,k}\| a_{k,n} \eta
            - \dsum_{k=j}^{m-1} \|D_{j,k}\| a_{k,n} \eta\\
         & \le \lbd_{j,n} \eta\\
         & \le \lbd \, a_{j,n} \eta.
      \end{split}
   \end{equation}
   From~\eqref{ine:||Vj,n_y||:hA_m,n y=0} and~\eqref{ine:||hAj,n_y||:A_m,n y=0}, we have
      $$ \eta \le \maxs{\lbd,\mu} \eta,$$
   and because $\maxs{\lbd,\mu} < 1,$ we must have $\eta = 0$. Hence
      $$ \dfrac{\|\hcA_{j,j}y\|}{a_{j,j}} = 0$$
   and this implies $y = \hcA_{j,j} y =0$.
\end{proof}

Now we are in conditions to prove Theorem~\ref{thm:robustness:Z}. Let $\widehat P_n = U_{n,n}$ and $\widehat Q_n = V_{n,n}$. By~\eqref{def:U_m,n} and~\eqref{def:V_m,n} we have
   $$ \widehat P_n + \widehat Q_n = P_n + Q_n = \Id$$
and from Lemmas~\ref{lemma:Um,jUj,n=Um,n} and~\ref{lemma:Vm,jVj,n=Vm,n} it follows immediately that $\widehat P_n$ and $\widehat Q_n$ is are projections. Moreover, by Lemmas~\ref{lemma:Um,jUj,n=Um,n} and~\ref{lemma:U_m,j_hcA_j,n_V_n,n_=_0}, we have
\begin{align*}
   \hcA_{m,n} \widehat P_n
   & = \widehat P_m \hcA_{m,n} \widehat P_n
      + \widehat Q_m \hcA_{m,n} \widehat P_n \\
   & = \widehat P_m \hcA_{m,n}  - \widehat P_m \hcA_{m,n} \widehat Q_n
      + \widehat Q_m U_{m,n} \\
   & = \widehat P_m \hcA_{m,n},
\end{align*}
i.e.,~\ref{eq:(S1)} is verified. Conditions~\ref{eq:(S2)} and~\ref{eq:(S3)} are immediate consequences of Lemmas~\ref{lemma:inclusions} and~\ref{lemma:invertibility}, respectively. Moreover, we have from Lemmas \ref{lemma:U_m,n_&_V_m,n} and~\ref{lemma:U_m,n+Vm,n_solutions} that
   $$ \|\hcA_{m,n} \widehat P_n\|
      = \|U_{m,n}\|
      \le \sigma a_{m,n}$$
and
   $$ \|\prts{\hcA_{m,n}|_{\widehat F_n}}^{-1} \widehat Q_m\|
      = \|V_{n,m} V_{m,m}\|
      = \|V_{n,m}\|
      \le \sigma b_{n,m}$$
for every $\prts{m,n} \in \Z^2_\ge$. Therefore equation~\eqref{eq:x_m+1=(A_m+B_m)x_m, m in Z} admits a general dichotomy with bounds $\prts{\sigma a_{m,n}}$ and $\prts{\sigma b_{m,n}}$.

\section{Proof of Theorem~\ref{thm:robustness:N}} \label{section:proof:robustness:N}

Suppose that equation~\eqref{eq:x_m+1=A_m x_m, m in N} has a general dichotomy with bounds $\prts{a_{m,n}}_{(m,n) \in \N^2_\ge}$ and $\prts{b_{m,n}}_{(m,n) \in \N^2_\le}$ and define
   $$ \tilde A_m =
      \begin{cases}
         A_m & \text{ if } m \in \N,\\
         \Id & \text{ if } m \in \Z\setminus\N,
      \end{cases}
      \ \ \ \text{ and } \ \ \
      \tilde P_m =
      \begin{cases}
         P_m & \text{ if } m \in \N,\\
         P_1 & \text{ if } m \in \Z\setminus\N.
      \end{cases}$$
Hence, using the projections $\tilde P_m$, $m \in \Z$, equation
   $$ x_{m+1} = \tilde A_m x_m, \ m \in \Z$$
has a general dichotomy with bounds $\prts{\tilde a_{m,n}}_{\prts{m,n} \in \Z^2_\ge}$ and $\prts{\tilde b_{m,n}}_{\prts{m,n} \in \Z^2_\le}$, where
   $$ \tilde a_{m,n} =
      \begin{cases}
         a_{m,n} & \text{ if } m \ge n \ge 1,\\
         a_{m,1} & \text{ if } m \ge 1 \ge n,\\
         a_{1,1} & \text{ if } 1 \ge m \ge n,
      \end{cases}
      \ \ \ \text{ and } \ \ \
      \tilde b_{m,n} =
      \begin{cases}
         b_{m,n} & \text{ if } 1 \le m \le n,\\
         b_{1,n} & \text{ if } m \le 1 \le n,\\
         b_{1,1} & \text{ if } m \le n \le 1.
      \end{cases}$$
Since we are assuming~\eqref{hyp:sup_m_a_m,n/a_m,j<infty:N}, the bounds $(\tilde a_{m,n})$ and $(\tilde b_{m,n})$ satisfy~\eqref{hyp:sup_m_a_m,n/a_m,j<infty} and~\eqref{hyp:sup_m_b_m,n/b_m,j<infty}.

Computing $\lbd_{m,n}$ and $\mu_{m,n}$ (see~\eqref{def:lbd_m,n} and~\eqref{def:mu_m,n}) for equation
\begin{equation}\label{eq:x_m+1=(tilde A_m+ tilde B_m)x_m, m in Z}
   x_{m+1} = \prts{\tilde A_m + \tilde B_m} x_m, \ m \in \Z,
\end{equation}
with
   $$ \tilde B_m =
      \begin{cases}
         B_m & \text{ if } m \in \N,\\
         0 & \text{ if } m \in \Z\setminus\N,
      \end{cases}$$
we have
   $$ \lbd_{m,n} =
      \begin{cases}
         \lbd'_{m,n} & \text{ if } m \ge n \ge 1,\\
         \lbd'_{m,1} & \text{ if } m \ge 1 \ge n,\\
         \lbd'_{1,1} & \text{ if } 1 \ge m \ge n,
      \end{cases}
      \ \ \ \text{ and } \ \ \
      \mu_{m,n} =
      \begin{cases}
         \mu'_{m,n} & \text{ if } 1 \le m \le n,\\
         \mu'_{1,n} & \text{ if } m \le 1 \le n,\\
         \mu'_{1,1} & \text{ if } m \le n \le 1,
      \end{cases}$$
where the $\lbd'$s and $\mu'$s are given by~\eqref{def:lbd'_m,n} and~\eqref{def:mu'_m,n}. Moreover,
\begin{align*}
   & \maxs{\sup_{\prts{m,n} \in \Z^2_\ge} \dfrac{\lbd_{m,n}}{\tilde a_{m,n}},\
         \sup_{\prts{m,n} \in \Z^2_\le} \dfrac{\mu_{m,n}}{\tilde b_{m,n}}}\\
   & = \maxs{\sup_{\prts{m,n} \in \N^2_\ge} \dfrac{\lbd'_{m,n}}{a_{m,n}}, \
         \sup_{\prts{m,n} \in \N^2_\le} \dfrac{\mu'_{m,n}}{b_{m,n}}}\\
   & = \theta < 1.
\end{align*}
Hence equation~\eqref{eq:x_m+1=(tilde A_m+ tilde B_m)x_m, m in Z} satisfies all the hypothesis in Theorem \ref{thm:robustness:Z} and this implies that has a general dichotomy with bounds $(\sigma \tilde a_{m,n})$ and $(\sigma \tilde b_{m,n})$, where $\sigma = 1/(1-\theta)$. Therefore equation~\eqref{eq:x_m+1=(A_m+B_m)x_m, m in N} has a general dichotomy with bounds $(\sigma a_{m,n})$ and $(\sigma b_{m,n})$ and Theorem~\ref{thm:robustness:N} is proved.

\section*{Acknowledgments}
   This work was partially supported by FCT though Centro de Ma\-te\-m\'a\-ti\-ca da Universidade da Beira Interior (project PEst-OE/MAT/UI0212/2011).
\bibliographystyle{elsart-num-sort}

\begin{thebibliography}{10}
\expandafter\ifx\csname url\endcsname\relax
  \def\url#1{\texttt{#1}}\fi
\expandafter\ifx\csname urlprefix\endcsname\relax\def\urlprefix{URL }\fi

\bibitem{Barreira-Fan-Valls-Zhang-TMNA-2011}
L.~Barreira, M.~Fan, C.~Valls, J.~Zhang, Robustness of nonuniform polynomial
  dichotomies for difference equations, Topol. Methods Nonlinear Anal. 37~(2)
  (2011) 357--376.

\bibitem{Barreira-Silva-Valls-JDE-2009}
L.~Barreira, C.~Silva, C.~Valls, Nonuniform behavior and robustness, J.
  Differential Equations 246~(9) (2009) 3579--3608.
\newline\urlprefix\url{http://dx.doi.org/10.1016/j.jde.2008.10.009}

\bibitem{Barreira-Valls-CMP-2005}
L.~Barreira, C.~Valls, Smoothness of invariant manifolds for nonautonomous
  equations, Comm. Math. Phys. 259~(3) (2005) 639--677.
\newline\urlprefix\url{http://dx.doi.org/10.1007/s00220-005-1380-z}

\bibitem{Barreira-Valls-JDE-2006-221}
L.~Barreira, C.~Valls, Stable manifolds for nonautonomous equations without
  exponential dichotomy, J. Differential Equations 221~(1) (2006) 58--90.
\newline\urlprefix\url{http://dx.doi.org/10.1016/j.jde.2005.04.005}

\bibitem{Barreira-Valls-JFA-2007}
L.~Barreira, C.~Valls, Conjugacies for linear and nonlinear perturbations of
  nonuniform behavior, J. Funct. Anal. 253~(1) (2007) 324--358.
\newline\urlprefix\url{http://dx.doi.org/10.1016/j.jfa.2007.05.007}

\bibitem{Barreira-Valls-JDE-2008}
L.~Barreira, C.~Valls, Robustness of nonuniform exponential dichotomies in
  {B}anach spaces, J. Differential Equations 244~(10) (2008) 2407--2447.
\newline\urlprefix\url{http://dx.doi.org/10.1016/j.jde.2008.02.028}

\bibitem{Barreira-Valls-NATMA-2010-72}
L.~Barreira, C.~Valls, Lyapunov sequences for exponential trichotomies,
  Nonlinear Anal. 72~(1) (2010) 192--203.
\newline\urlprefix\url{http://dx.doi.org/10.1016/j.na.2009.06.059}

\bibitem{Barreira-Valls-DCDS-A-2012-32-12}
L.~Barreira, C.~Valls, Noninvertible cocycles: robustness of exponential
  dichotomies, Discrete Contin. Dyn. Syst. 32~(12) (2012) 4111--4131.
\newline\urlprefix\url{http://dx.doi.org/10.3934/dcds.2012.32.4111}

\bibitem{Bento-Silva-JFA-2009}
A.~J.~G. Bento, C.~Silva, Stable manifolds for nonuniform polynomial
  dichotomies, J. Funct. Anal. 257~(1) (2009) 122--148.
\newline\urlprefix\url{http://dx.doi.org/10.1016/j.jfa.2009.01.032}

\bibitem{Bento-Silva-JDDE}
A.~J.~G. Bento, C.~M. Silva, Generalized nonuniform dichotomies and local
  stable manifolds, J. Dynam. Differential Equations, to appear, Preprint
  arXiv:1007.5039 [math.DS].
\newline\urlprefix\url{http://arxiv.org/abs/1007.5039}

\bibitem{Bento-Silva-2012-arXiv1209.6589B}
A.~J.~G. {Bento}, C.~M. {Silva}, {Nonuniform dichotomic behavior: Lipschitz
  invariant manifolds for difference equations},\ Preprint arXiv:1209.6589
  [math.DS],\ submitted.
\newline\urlprefix\url{http://arxiv.org/abs/1209.6589}

\bibitem{Bento-Silva-2012-arXiv1210.7740B}
A.~J.~G. {Bento}, C.~M. {Silva}, {Nonuniform dichotomic behavior: Lipschitz
  invariant manifolds for ODEs}, ArXiv e-prints,\ Preprint arXiv1210.7740
  [math.DS],\ submitted.
\newline\urlprefix\url{http://arxiv.org/abs/1210.7740}

\bibitem{Bento-Silva-NATMA-2012}
A.~J.~G. Bento, C.~M. Silva, Nonuniform {$(\mu,\nu)$}-dichotomies and local
  dynamics of difference equations, Nonlinear Anal. 75~(1) (2012) 78--90.
\newline\urlprefix\url{http://dx.doi.org/10.1016/j.na.2011.08.008}

\bibitem{Chang-Zhang-Qin-JMAA-2012}
X.~Chang, J.~Zhang, J.~Qin, Robustness of nonuniform {$(\mu,\nu)$}-dichotomies
  in {B}anach spaces, J. Math. Anal. Appl. 387~(2) (2012) 582--594.
\newline\urlprefix\url{http://dx.doi.org/10.1016/j.jmaa.2011.09.026}

\bibitem{Chow-Leiva-JDE-1995}
S.-N. Chow, H.~Leiva, Existence and roughness of the exponential dichotomy for
  skew-product semiflow in {B}anach spaces, J. Differential Equations 120~(2)
  (1995) 429--477.
\newline\urlprefix\url{http://dx.doi.org/10.1006/jdeq.1995.1117}

\bibitem{Chu-BSM-2013}
J.~Chu, Robustness of nonuniform behavior for discrete dynamics, Bull. Sci.
  Math., (2013).
\newline\urlprefix\url{http://dx.doi.org/10.1016/j.bulsci.2013.03.003}

\bibitem{Coppel-JDE-1967}
W.~A. Coppel, Dichotomies and reducibility, J. Differential Equations 3 (1967)
  500--521.
\newline\urlprefix\url{http://dx.doi.org/10.1016/0022-0396(67)90014-9}

\bibitem{Daleckii-Krein-TMM-1974}
J.~L. Dalec{\cprime}ki{\u\i}, M.~G. Kre{\u\i}n, Stability of solutions of
  differential equations in {B}anach space, American Mathematical Society,
  Providence, R.I., 1974, translated from the Russian by S. Smith, Translations
  of Mathematical Monographs, Vol. 43.

\bibitem{Jiang-EJDE-2012}
Y.~Jiang, Robustness of a nonuniform {$(\mu,\nu)$} trichotomy in {B}anach
  spaces, Electron. J. Differential Equations (2012) No. 154, 11.
\newline\urlprefix\url{http://ejde.math.unt.edu/Volumes/2012/154/jiang.pdf}

\bibitem{Li-AM-1934}
T.~Li, Die {S}tabilit\"atsfrage bei {D}ifferenzengleichungen, Acta Math. 63~(1)
  (1934) 99--141.
\newline\urlprefix\url{http://dx.doi.org/10.1007/BF02547352}

\bibitem{Massera-Schaffer-AM-1958}
J.~L. Massera, J.~J. Sch{\"a}ffer, Linear differential equations and functional
  analysis. {I}, Ann. of Math. (2) 67 (1958) 517--573.
\newline\urlprefix\url{http://www.jstor.org/stable/1969871}

\bibitem{Naulin-Pinto-NATMA-1994}
R.~Naulin, M.~Pinto, Dichotomies and asymptotic solutions of nonlinear
  differential systems, Nonlinear Anal. 23~(7) (1994) 871--882.
\newline\urlprefix\url{http://dx.doi.org/10.1016/0362-546X(94)90125-2}

\bibitem{Naulin-Pinto-JDE-1995}
R.~Naulin, M.~Pinto, Roughness of {$(h,k)$}-dichotomies, J. Differential
  Equations 118~(1) (1995) 20--35.
\newline\urlprefix\url{http://dx.doi.org/10.1006/jdeq.1995.1065}

\bibitem{Perron-MZ-1930}
O.~Perron, Die {S}tabilit\"atsfrage bei {D}ifferentialgleichungen, Math. Z.
  32~(1) (1930) 703--728.
\newline\urlprefix\url{http://dx.doi.org/10.1007/BF01194662}

\bibitem{Pliss-Sell-JDDE-1999}
V.~A. Pliss, G.~R. Sell, Robustness of exponential dichotomies in
  infinite-dimensional dynamical systems, J. Dynam. Differential Equations
  11~(3) (1999) 471--513.
\newline\urlprefix\url{http://dx.doi.org/10.1023/A:1021913903923}

\bibitem{Popescu-JMAA-2006}
L.~H. Popescu, Exponential dichotomy roughness on {B}anach spaces, J. Math.
  Anal. Appl. 314~(2) (2006) 436--454.
\newline\urlprefix\url{http://dx.doi.org/10.1016/j.jmaa.2005.04.011}

\bibitem{Preda-Megan-BAusMS-1983}
P.~Preda, M.~Megan, Nonuniform dichotomy of evolutionary processes in {B}anach
  spaces, Bull. Austral. Math. Soc. 27~(1) (1983) 31--52.
\newline\urlprefix\url{http://dx.doi.org/10.1017/S0004972700011473}

\end{thebibliography}
\def\cprime{$'$}

\end{document}